\documentclass[reqno,oneside,12pt]{amsart}

\usepackage[T1]{fontenc}
\usepackage{times, mathptm, graphicx, xcolor}
\usepackage{tikz}
\input{xy}
\xyoption{all}

\usepackage{amssymb,epsfig,verbatim,xypic}

\theoremstyle{plain}

\newtheorem{thm}{Theorem}[section]

\newtheorem{cor}[thm]{Corollary}
\newtheorem{pro}[thm]{Proposition}
\newtheorem{lem}[thm]{Lemma}
\newtheorem{proposition-principale}[thm]{Proposition principale}
\newtheorem{thm-principal}{Th\'eor\`eme principal}[section]

\newtheorem{mthm}{Theorem}

\theoremstyle{definition}

\newtheorem{eg}[thm]{Example}
\newtheorem{rem}[thm]{Remark}

\newtheorem{Corollary}[thm]{Corollary}

\newtheorem*{thm*}{Theorem}

\theoremstyle{definition}

\newtheorem{Definition}[thm]{Definition}
\newtheorem{Remark}[thm]{Remark}

\newtheorem{Example}[thm]{Example}

\def\C{\mathbf{C}}
\def\bfk{\mathbf{k}}

\def\R{\mathbf{R}}
\def\Q{\mathbf{Q}}

\def\Z{\mathbf{Z}}
\def\N{\mathbf{N}}

\def\bfx{{\mathbf{x}}}
\def\bfy{{\mathbf{y}}}

\def\bbP{\mathbb{P}}
\def\bbA{\mathbb{A}}

\def\Aut{{\sf{Aut}}}
\def\Bir{{\sf{Bir}}}

\def\PGL{{\sf{PGL}}\,}

\def\GL{{\sf{GL}}\,}

\addtocounter{section}{0}             
\numberwithin{equation}{section}       

  \newcommand{\serge}[1]{{\color{red}*}\marginpar{\tiny  \color{red} SC: #1}}

\begin{document}

\setlength{\baselineskip}{0.54cm}        
%
%
\title[Algebraic growth of the Cremona group]
{Algebraic growth of the Cremona group}
\date{2024}

\author[Alberto Calabri]{Alberto Calabri}
\address{\sc Alberto Calabri\\ Dipartimento di Matematica e Informatica, Universit\`a di Ferrara, Via Machiavelli 30, 44121 Ferrara, Italy}
\email{clblrt@unife.it}

\author[Serge Cantat]{Serge Cantat}
\address{\sc Serge Cantat\\
IRMAR (UMR 6625 du CNRS), Universit{\'e} de Rennes 1, Campus de Beaulieu, 35042 Rennes cedex, France} 
\email{serge.cantat@univ-rennes1.fr}

\author[Alex Massarenti]{Alex Massarenti}
\address{\sc Alex Massarenti\\ Dipartimento di Matematica e Informatica, Universit\`a di Ferrara, Via Machiavelli 30, 44121 Ferrara, Italy}
\email{msslxa@unife.it}

\author[Fran\c cois Maucourant]{Fran\c cois Maucourant}
\address{\sc Fran\c cois Maucourant\\ IRMAR (UMR 6625 du CNRS), Universit{\'e} de Rennes 1, Campus de Beaulieu, 35042 Rennes cedex, France}
\email{francois.maucourant@univ-rennes1.fr}

\author[Massimiliano Mella]{Massimiliano Mella}
\address{\sc Massimiliano Mella\\ Dipartimento di Matematica e Informatica, Universit\`a di Ferrara, Via Machiavelli 30, 44121 Ferrara, Italy}
\email{mll@unife.it}

%
%

%
%

%
%

\date{\today}

\begin{abstract} 
We initiate the study of the ``{\emph{algebraic growth}}'' of groups of automorphisms and birational transformations of algebraic varieties.
Our main result concerns $\Bir(\bbP^2)$, the Cremona group in $2$ variables. This group is the union, for all degrees $d\geq 1$, of the algebraic variety $\Bir(\bbP^2)_d$ of birational transformations of the plane of degree $d$. Let $N_d$ denote the number of irreducible components of $\Bir(\bbP^2)_d$. We describe the asymptotic growth of $N_d$ as $d$ goes to $+\infty$, showing that there are two constants $A$ and $B>0$ such that 
$$
A\sqrt{\ln(d)} \leq \ln \left(\ln \left(\sum_{e\leq d} N_e \right) \right) \leq B \sqrt{\ln(d)}
$$
for all large enough degrees $d$. This growth type seems quite unusual and shows that computing the algebraic growth of $\Bir(\bbP^2)$ is a challenging problem in general.   
\end{abstract}

\maketitle
 %
%

\setcounter{tocdepth}{1}
\tableofcontents 
 
\vfill
\pagebreak

\section{Introduction}

\subsection{Automorphisms of the affine plane} Consider the group $\Aut(\bbA^2_\bfk)$ of automorphisms of the affine plane $\bbA^2_\bfk$, over some field $\bfk$. For each degree $d\geq 1$, Friedland and Milnor proved in \cite{friedland-milnor} that the automorphisms of degree $d$ form an algebraic variety $\Aut(\bbA^2_\bfk)_d$  and   the number of  
 irreducible components of $\Aut(\bbA^2_\bfk)_d$ is equal to $K_d$,
 where $K$ is the so called Kalm\'ar's function; that is, $K_d$ is the
 number of {\emph{ordered factorizations $d=d_1\cdot d_2 \cdots d_s$ 
 in integers $d_i\geq 2$}}; that is, two factorizations are considered identical if they contain the same integers $d_i$ written in the same order. For instance $K_6=3$ because $6=2\cdot 3=3\cdot 2 = 6$. Then, using the results of Del\'eglise, Hernane, and Nicolas described in~\cite{deleglise-hernane-nicolas}, we obtain the following theorem. To state it, we denote by  $\zeta(s)=\sum_{n\geq 1} n^{-s}$  the Riemann zeta function and by $\rho\simeq 1.728$ the real number  defined by $\zeta(\rho)=2$.
 
\begin{mthm}\label{ThA} The number $K_d$ of irreducible components of $\Aut(\bbA^2_\bfk)_d$ satisfies the following properties:
\begin{enumerate}
\item as $n$ goes to $\infty$, 
$$\sum_{d=1}^n K_d \simeq a  n^{\rho}
$$ 
with $\zeta(\rho)=2$ and $a= \vert \rho \zeta'(\rho)\vert^{-1}$;
\item $K_n\geq 1$ for all $n$ and $K_p=1$ if and only if  $p$ is prime;  in particular 
$$\liminf_{d} \log(K_d)=0;$$
\item $\limsup_d\log(K_d)/\log(d)=\rho$.
\end{enumerate} 
\end{mthm}

Thus, we see that $K_d$ oscillates, with minimal values equal to $1$ and maximal values growing polynomially, like $c^{st} d^\rho$. 

\subsection{The Cremona group} Our goal is to study a similar question for birational transformations of the plane $\bbP^2_\bfk$, over an algebraically closed field $\bfk$. As will be explained in Section~\ref{Bird}, the birational transformations of $\bbP^2_\bfk$ of degree $d$ form an algebraic variety $\Bir(\bbP^2_\bfk)_d$, and we shall denote by $N_d$ the number of its irreducible components. The main result of this article is the following theorem. 

\begin{mthm}\label{ThB} 
The number $N_d$ of irreducible components of $\Bir(\bbP^2_\bfk)_d$ satisfies 
$$ 
0.832\simeq\sqrt{\ln(2)} \leq \limsup_{d\to +\infty}\frac{\ln(\ln(N_d))}{\sqrt{\ln(d)}}\leq 2 \sqrt{\ln(2)} \simeq 1.665.
$$
For every $\epsilon>0$, there is an integer $D(\epsilon)$ such that the sum $N_{\leq d}:=\sum_{d'=1}^d  N_{d'}$ satisfies 
 $$ 
 \sqrt{\ln(2)} -\epsilon \leq \frac{\ln (\ln (N_{\leq d}))}{\sqrt{\ln d}} \leq 2\sqrt{\ln 2} +\epsilon
 $$
 for all $d\geq D(\epsilon)$. \end{mthm}

In particular, the growth is subexponential but faster than any polynomial function of~$d$, so we observe a phenomenon of intermediate asymptotic growth.

\subsection{} We computed the values of $N_d$ for $d\in [1,249]$. The first ones are given by 
\begin{equation*}
\begin{array}{c|c|c|c|c|c|c|c|c|c|clclclclclclclclc}
d &   1 & 2 & 3 & 4 & 5 & 6 & 7 & 8 & 9 & 10  & 11 & 12 & 13 & 14 & 15 & 16 & 17 & 18  \\ 
\hline
N_d &  1 & 1 & 1 & 2 & 5 &  4 &  5 & 9 & 10  & 17  & 19 & 29 & 34 & 51 & 63 & 88 & 102 & 152 
\end{array}
\end{equation*}
and to get a rough idea of the growth rate one can consider the following list:
$$
\begin{array}{c|c|c|c|c|c|c|c|c|c|c}
d &    50 & 100 & 125 & 150 & 200 & 249 \\ 
\hline
N_d &   52683 & 9733297 & 59637891 & 287772117 &  3585742777 &  25275093795 
\end{array} 
$$
Now, set $c(d) = \ln(\ln(N_d))/\sqrt{\ln(d)}$. In the  picture below, the black curve is made of the points $\left(d,c(d)\right)$ for $d\in [5,249]$. 
The red segments are at height $\sqrt{\ln(2)}$ and $2\sqrt{\ln(2)}$. As we can see $c(d)$ lies in between $\sqrt{\ln(2)}$ and $2\sqrt{\ln(2)}$ when  $16\leq d\leq 249$; and for $d$ large,
$c(d)\simeq 1.354$ is closer to the upper bound. 
$$
\!\!\!\!\!\!\!\!\begin{tikzpicture}
\draw[thin,->] (0,0) -- (13,0) node[right] {$d$};
\draw[thin,->] (0,0) -- (0,2) node[above] {$c(d)$};

\draw[domain=0:13, smooth, variable=\x, red] plot ({\x}, {0.832554611157697756353164644895});

\draw[domain=0:13, smooth, variable=\x, red] plot ({\x}, {1.66510922231539551270632928979});

\foreach \point in {( 0.250000000000000000000000000000, 0.0741330135179314042510133819544 ),
    ( 0.300000000000000000000000000000, 0.244018009875873667193196270177 ),
    ( 0.350000000000000000000000000000, 0.341146264439434424621359494948 ),
    ( 0.400000000000000000000000000000, 0.545894820333207989126776284434 ),
    ( 0.450000000000000000000000000000, 0.562659585101620244541791715281 ),
    ( 0.500000000000000000000000000000, 0.686300847406593547371546647875 ),
    ( 0.550000000000000000000000000000, 0.697390126952078290683837266745 ),
    ( 0.600000000000000000000000000000, 0.770199081373423281300529488815 ),
    ( 0.650000000000000000000000000000, 0.786906141556748942685311960414 ),
    ( 0.700000000000000000000000000000, 0.842775653737346011692346490513 ),
    ( 0.750000000000000000000000000000, 0.863781805010748655482688333137 ),
    ( 0.800000000000000000000000000000, 0.900258306495017828683070625220 ),
    ( 0.850000000000000000000000000000, 0.909848313174003243730162109129 ),
    ( 0.900000000000000000000000000000, 0.949469902925946928138663551194 ),
    ( 0.950000000000000000000000000000, 0.954217375762054141181522376787 ),
    ( 1.00000000000000000000000000000, 0.983774358656829345882327375753 ),
    ( 1.05000000000000000000000000000, 0.992687085676186017852397909382 ),
    ( 1.10000000000000000000000000000, 1.01342438171825204399672872715 ),
    ( 1.15000000000000000000000000000, 1.01956077531307299864675023734 ),
    ( 1.20000000000000000000000000000, 1.04345438749958292383356061628 ),
    ( 1.25000000000000000000000000000, 1.04455689536209407186352424262 ),
    ( 1.30000000000000000000000000000, 1.06326808836763217775597738558 ),
    ( 1.35000000000000000000000000000, 1.07066061758833393616741670051 ),
    ( 1.40000000000000000000000000000, 1.08451598827447132070872131099 ),
    ( 1.45000000000000000000000000000, 1.08686623761655886326780831737 ),
    ( 1.50000000000000000000000000000, 1.10342656902489200555176575813 ),
    ( 1.55000000000000000000000000000, 1.10411474163002087134473391007 ),
    ( 1.60000000000000000000000000000, 1.11692255872896492029950433366 ),
    ( 1.65000000000000000000000000000, 1.12138097456187570145391702212 ),
    ( 1.70000000000000000000000000000, 1.13092451550933074881957849379 ),
    ( 1.75000000000000000000000000000, 1.13371940681369704280352414052 ),
    ( 1.80000000000000000000000000000, 1.14532075110948012500293594163 ),
    ( 1.85000000000000000000000000000, 1.14567944909178899411395854354 ),
    ( 1.90000000000000000000000000000, 1.15512511936542961677008604099 ),
    ( 1.95000000000000000000000000000, 1.15846419181131085885957683150 ),
    ( 2.00000000000000000000000000000, 1.16572669006844565034233014871 ),
    ( 2.05000000000000000000000000000, 1.16720592459783235389212052922 ),
    ( 2.10000000000000000000000000000, 1.17617828449248232718825752719 ),
    ( 2.15000000000000000000000000000, 1.17683552127628967365109976062 ),
    ( 2.20000000000000000000000000000, 1.18390023278889533540785974971 ),
    ( 2.25000000000000000000000000000, 1.18621817283552595355384350699 ),
    ( 2.30000000000000000000000000000, 1.19202021153246114835447332527 ),
    ( 2.35000000000000000000000000000, 1.19332604764296498715221570199 ),
    ( 2.40000000000000000000000000000, 1.20030550207426539127659034817 ),
    ( 2.45000000000000000000000000000, 1.20053887669903413437101720221 ),
    ( 2.50000000000000000000000000000, 1.20643867150376061787332778055 ),
    ( 2.55000000000000000000000000000, 1.20817670085176171910805394579 ),
    ( 2.60000000000000000000000000000, 1.21280457056890387089258654340 ),
    ( 2.65000000000000000000000000000, 1.21381704534180209910839561654 ),
    ( 2.70000000000000000000000000000, 1.21944255569614953844517649222 ),
    ( 2.75000000000000000000000000000, 1.21985726107598302862103847346 ),
    ( 2.80000000000000000000000000000, 1.22446765893594483083216144981 ),
    ( 2.85000000000000000000000000000, 1.22574202641897481194653536475 ),
    ( 2.90000000000000000000000000000, 1.22979021980705040778888602016 ),
    ( 2.95000000000000000000000000000, 1.23052801532150901843785752168 ),
    ( 3.00000000000000000000000000000, 1.23514702092224175742963134493 ),
    ( 3.05000000000000000000000000000, 1.23537870711957894891798288153 ),
    ( 3.10000000000000000000000000000, 1.23938739088148960517882090306 ),
    ( 3.15000000000000000000000000000, 1.24039355501634238641180579532 ),
    ( 3.20000000000000000000000000000, 1.24378672536370707529411584348 ),
    ( 3.25000000000000000000000000000, 1.24434923095775230630135007433 ),
    ( 3.30000000000000000000000000000, 1.24822450738341644108682641961 ),
    ( 3.35000000000000000000000000000, 1.24845591979174665855689294295 ),
    ( 3.40000000000000000000000000000, 1.25184589286221837668306222913 ),
    ( 3.45000000000000000000000000000, 1.25257008936780690410523560750 ),
    ( 3.50000000000000000000000000000, 1.25561127831984611596353167852 ),
    ( 3.55000000000000000000000000000, 1.25604402708843977894638408915 ),
    ( 3.60000000000000000000000000000, 1.25929625093417659097365948668 ),
    ( 3.65000000000000000000000000000, 1.25952189508652427185268857989 ),
    ( 3.70000000000000000000000000000, 1.26248552187491602763005371539 ),
    ( 3.75000000000000000000000000000, 1.26302514804897354054512248886 ),
    ( 3.80000000000000000000000000000, 1.26563839658838527579812645304 ),
    ( 3.85000000000000000000000000000, 1.26597899746624466774680005450 ),
    ( 3.90000000000000000000000000000, 1.26888361123745909280161081942 ),
    ( 3.95000000000000000000000000000, 1.26900722060823291232284696587 ),
    ( 4.00000000000000000000000000000, 1.27160725961084183639306748707 ),
    ( 4.05000000000000000000000000000, 1.27200098545509076618486951062 ),
    ( 4.10000000000000000000000000000, 1.27439813462942649836639194708 ),
    ( 4.15000000000000000000000000000, 1.27463426181571884164794920842 ),
    ( 4.20000000000000000000000000000, 1.27711229922864775174305748168 ),
    ( 4.25000000000000000000000000000, 1.27726092701656798794736523813 ),
    ( 4.30000000000000000000000000000, 1.27957598947084611405975781638 ),
    ( 4.35000000000000000000000000000, 1.27982723775549358333746994179 ),
    ( 4.40000000000000000000000000000, 1.28198388798040283125326976320 ),
    ( 4.45000000000000000000000000000, 1.28217158192259326909928948369 ),
    ( 4.50000000000000000000000000000, 1.28441015758523884265233826234 ),
    ( 4.55000000000000000000000000000, 1.28447864434314714096508816369 ),
    ( 4.60000000000000000000000000000, 1.28653460240226340422560428646 ),
    ( 4.65000000000000000000000000000, 1.28674564712357300749037853132 ),
    ( 4.70000000000000000000000000000, 1.28870855651254582083579961075 ),
    ( 4.75000000000000000000000000000, 1.28881929671821745316848865665 ),
    ( 4.80000000000000000000000000000, 1.29078648684901715011528070473 ),
    ( 4.85000000000000000000000000000, 1.29085503147793866986719427819 ),
    ( 4.90000000000000000000000000000, 1.29276663417176972855468507378 ),
    ( 4.95000000000000000000000000000, 1.29284611076132884143923430486 ),
    ( 5.00000000000000000000000000000, 1.29464493190228359019757203711 ),
    ( 5.05000000000000000000000000000, 1.29471976608456213455002708265 ),
    ( 5.10000000000000000000000000000, 1.29651399331754449603987530776 ),
    ( 5.15000000000000000000000000000, 1.29654752618059625276695988855 ),
    ( 5.20000000000000000000000000000, 1.29825228573692462228383257933 ),
    ( 5.25000000000000000000000000000, 1.29831752076060265074800163565 ),
    ( 5.30000000000000000000000000000, 1.29998366033150597801217310115 ),
    ( 5.35000000000000000000000000000, 1.29998775760290649812942170286 ),
    ( 5.40000000000000000000000000000, 1.30161225172490865825093140604 ),
    ( 5.45000000000000000000000000000, 1.30164862536923474284710529216 ),
    ( 5.50000000000000000000000000000, 1.30323660271693988495311239672 ),
    ( 5.55000000000000000000000000000, 1.30319987880642298690362253650 ),
    ( 5.60000000000000000000000000000, 1.30473816618744702899380197681 ),
    ( 5.65000000000000000000000000000, 1.30475867551839973840790015113 ),
    ( 5.70000000000000000000000000000, 1.30624059889517648115894737433 ),
    ( 5.75000000000000000000000000000, 1.30622205472535567673043166570 ),
    ( 5.80000000000000000000000000000, 1.30767425481100409872191319152 ),
    ( 5.85000000000000000000000000000, 1.30763765251127765554726070604 ),
    ( 5.90000000000000000000000000000, 1.30909387111798916933295731775 ),
    ( 5.95000000000000000000000000000, 1.30903675858103761156431000532 ),
    ( 6.00000000000000000000000000000, 1.31039935421448287130212784440 ),
    ( 6.05000000000000000000000000000, 1.31038766036642850924973594483 ),
    ( 6.10000000000000000000000000000, 1.31174120995024847716094695428 ),
    ( 6.15000000000000000000000000000, 1.31164735594147678215144563037 ),
    ( 6.20000000000000000000000000000, 1.31299230175994411933454212730 ),
    ( 6.25000000000000000000000000000, 1.31295561435451901599739316916 ),
    ( 6.30000000000000000000000000000, 1.31421003509183717601845637495 ),
    ( 6.35000000000000000000000000000, 1.31415616041771576897695907608 ),
    ( 6.40000000000000000000000000000, 1.31542210453842961117366502342 ),
    ( 6.45000000000000000000000000000, 1.31531996915502391775442254137 ),
    ( 6.50000000000000000000000000000, 1.31659569355875319144995543428 ),
    ( 6.55000000000000000000000000000, 1.31649046605089577291465863982 ),
    ( 6.60000000000000000000000000000, 1.31765443483931398236431801181 ),
    ( 6.65000000000000000000000000000, 1.31763104198110173821541814411 ),
    ( 6.70000000000000000000000000000, 1.31881437431428893154207368587 ),
    ( 6.75000000000000000000000000000, 1.31865780030914925488326373952 ),
    ( 6.80000000000000000000000000000, 1.31985193239449206368097001618 ),
    ( 6.85000000000000000000000000000, 1.31976600051992525361353429571 ),
    ( 6.90000000000000000000000000000, 1.32086131476368401585388713446 ),
    ( 6.95000000000000000000000000000, 1.32078071681757245238712108003 ),
    ( 7.00000000000000000000000000000, 1.32189365234996259268633833107 ),
    ( 7.05000000000000000000000000000, 1.32174656042017951830562900336 ),
    ( 7.10000000000000000000000000000, 1.32287631703805792953833583633 ),
    ( 7.15000000000000000000000000000, 1.32275373572552732224809275601 ),
    ( 7.20000000000000000000000000000, 1.32377070615693335660635714765 ),
    ( 7.25000000000000000000000000000, 1.32371695231106498369064726147 ),
    ( 7.30000000000000000000000000000, 1.32475711933141346454373896241 ),
    ( 7.35000000000000000000000000000, 1.32456223931028411024681880201 ),
    ( 7.40000000000000000000000000000, 1.32564520772720552705003565934 ),
    ( 7.45000000000000000000000000000, 1.32552924664666764764101069462 ),
    ( 7.50000000000000000000000000000, 1.32648753112799900192235301498 ),
    ( 7.55000000000000000000000000000, 1.32638576933308894275959014009 ),
    ( 7.60000000000000000000000000000, 1.32737781710607697041971873995 ),
    ( 7.65000000000000000000000000000, 1.32720507106825736493101584825 ),
    ( 7.70000000000000000000000000000, 1.32822665185463785514291707476 ),
    ( 7.75000000000000000000000000000, 1.32807329670835703574600086639 ),
    ( 7.80000000000000000000000000000, 1.32896918153508296248382683903 ),
    ( 7.85000000000000000000000000000, 1.32889215470054440493522009878 ),
    ( 7.90000000000000000000000000000, 1.32983453902045149795632185670 ),
    ( 7.95000000000000000000000000000, 1.32961622682848416962769491145 ),
    ( 8.00000000000000000000000000000, 1.33059372372609923855925054279 ),
    ( 8.05000000000000000000000000000, 1.33045727409634402445496468012 ),
    ( 8.10000000000000000000000000001, 1.33130323381957599643551464702 ),
    ( 8.14999999999999999999999999999, 1.33119615091082912103259804758 ),
    ( 8.20000000000000000000000000000, 1.33208763851125492512126706850 ),
    ( 8.25000000000000000000000000000, 1.33189010636037876925143039089 ),
    ( 8.30000000000000000000000000000, 1.33281569671652794239409705264 ),
    ( 8.35000000000000000000000000001, 1.33264133496918282504051764921 ),
    ( 8.39999999999999999999999999999, 1.33345607989370838884030658438 ),
    ( 8.45000000000000000000000000000, 1.33336281256728066195766023440 ),
    ( 8.50000000000000000000000000000, 1.33421230535398232767513571323 ),
    ( 8.55000000000000000000000000000, 1.33397166784836884357361023408 ),
    ( 8.60000000000000000000000000001, 1.33486523076383676220830997395 ),
    ( 8.64999999999999999999999999999, 1.33471894313435248883546832873 ),
    ( 8.70000000000000000000000000000, 1.33547950177212396528634514376 ),
    ( 8.75000000000000000000000000000, 1.33535641894056852932259039170 ),
    ( 8.80000000000000000000000000000, 1.33616671288010712870896300685 ),
    ( 8.85000000000000000000000000001, 1.33594906207892193900026760768 ),
    ( 8.89999999999999999999999999999, 1.33680380440364156251025399235 ),
    ( 8.95000000000000000000000000000, 1.33662126431768220744467212984 ),
    ( 9.00000000000000000000000000000, 1.33735335754497065553652082709 ),
    ( 9.05000000000000000000000000000, 1.33724644281477977506161568019 ),
    ( 9.10000000000000000000000000001, 1.33801758820885261059686176287 ),
    ( 9.14999999999999999999999999999, 1.33777410237564636746882341900 ),
    ( 9.20000000000000000000000000000, 1.33859621463508763211431050705 ),
    ( 9.25000000000000000000000000000, 1.33843258666263879740073663837 ),
    ( 9.30000000000000000000000000000, 1.33912137331069065907294191266 ),
    ( 9.35000000000000000000000000001, 1.33898639113457231810105938082 ),
    ( 9.39999999999999999999999999999, 1.33973752356979663703952602247 ),
    ( 9.45000000000000000000000000000, 1.33951025400918080480611524331 ),
    ( 9.50000000000000000000000000000, 1.34029633145430157094085676778 ),
    ( 9.55000000000000000000000000000, 1.34009929734290245689901521771 ),
    ( 9.60000000000000000000000000001, 1.34076586820075575091948056969 ),
    ( 9.64999999999999999999999999999, 1.34065563294131075798074049776 ),
    ( 9.70000000000000000000000000000, 1.34136485866058897394132483578 ),
    ( 9.75000000000000000000000000000, 1.34111067541397994675814278636 ),
    ( 9.80000000000000000000000000000, 1.34186945961280500730656465613 ),
    ( 9.85000000000000000000000000001, 1.34169734118211334668927130651 ),
    ( 9.89999999999999999999999999999, 1.34233184899525160369504892889 ),
    ( 9.95000000000000000000000000000, 1.34219099070439733805331492561 ),
    ( 10.0000000000000000000000000000, 1.34288079778160195944345758786 ),
    ( 10.0500000000000000000000000000, 1.34264217855749036701715632710 ),
    ( 10.1000000000000000000000000000, 1.34336932749710671178598022536 ),
    ( 10.1500000000000000000000000000, 1.34317442492280807243766499545 ),
    ( 10.2000000000000000000000000000, 1.34378738112406484380240205686 ),
    ( 10.2500000000000000000000000000, 1.34366577257415376000903416027 ),
    ( 10.3000000000000000000000000000, 1.34431911626040698437475002484 ),
    ( 10.3500000000000000000000000000, 1.34405863520697332599865619973 ),
    ( 10.4000000000000000000000000000, 1.34477295660866311216299904675 ),
    ( 10.4500000000000000000000000000, 1.34459210656887328045460395876 ),
    ( 10.5000000000000000000000000000, 1.34517508494532382078584428066 ),
    ( 10.5500000000000000000000000000, 1.34502465591228217723739527789 ),
    ( 10.6000000000000000000000000000, 1.34566346856464221613277012881 ),
    ( 10.6500000000000000000000000000, 1.34542756085222550365162739708 ),
    ( 10.7000000000000000000000000000, 1.34610488239461232430695669782 ),
    ( 10.7500000000000000000000000000, 1.34590388261498319989033951288 ),
    ( 10.8000000000000000000000000000, 1.34646787599625414100360241450 ),
    ( 10.8500000000000000000000000000, 1.34633943563764696880284161548 ),
    ( 10.9000000000000000000000000000, 1.34695136149806701333540993302 ),
    ( 10.9500000000000000000000000000, 1.34668897903143190158114848341 ),
    ( 11.0000000000000000000000000000, 1.34735348718715096994109449735 ),
    ( 11.0500000000000000000000000000, 1.34716689792742592847433529770 ),
    ( 11.1000000000000000000000000000, 1.34770467905351514404203858264 ),
    ( 11.1500000000000000000000000000, 1.34755639241449610444356906781 ),
    ( 11.2000000000000000000000000000, 1.34815458403602351754713947138 ),
    ( 11.2500000000000000000000000000, 1.34791175270783243374965448224 ),
    ( 11.3000000000000000000000000000, 1.34854353423625851299268607368 ),
    ( 11.3500000000000000000000000000, 1.34833717048031748564220720213 ),
    ( 11.4000000000000000000000000000, 1.34886781064491866274322892532 ),
    ( 11.4500000000000000000000000000, 1.34873477165000108016742360931 ),
    ( 11.5000000000000000000000000000, 1.34930361227115452326425409185 ),
    ( 11.5500000000000000000000000000, 1.34903886482072723561173684221 ),
    ( 11.6000000000000000000000000000, 1.34966169224618968277683277399 ),
    ( 11.6500000000000000000000000000, 1.34947767387393386986656599961 ),
    ( 11.7000000000000000000000000000, 1.34997813616600265083624424730 ),
    ( 11.7500000000000000000000000000, 1.34982485999861849886359983369 ),
    ( 11.8000000000000000000000000000, 1.35038122545390116019545827832 ),
    ( 11.8500000000000000000000000000, 1.35013704319657462421868277936 ),
    ( 11.9000000000000000000000000000, 1.35073586402552192050517249714 ),
    ( 11.9500000000000000000000000000, 1.35052861799360421435708826853 ),
    ( 12.0000000000000000000000000000, 1.35102253749172402362539439453 ),
    ( 12.0500000000000000000000000000, 1.35088206639548138192216197754 ),
    ( 12.1000000000000000000000000000, 1.35141622282101049380954787258 ),
    ( 12.1500000000000000000000000000, 1.35115684240009054527011560028 ),
    ( 12.2000000000000000000000000000, 1.35174383655624045264301243658 ),
    ( 12.2500000000000000000000000000, 1.35155501693693038425118047255 ),
    ( 12.3000000000000000000000000000, 1.35202048570871007199182494277 ),
    ( 12.3500000000000000000000000000, 1.35186369279049796679275570930 ),
    ( 12.4000000000000000000000000000, 1.35239209380900999853161139740 ),
    ( 12.4500000000000000000000000000, 1.35214871399979048328914637308 )}{
    \fill \point circle (0.5pt);
}
\end{tikzpicture}
$$

\subsection{Algebraic growth type} These results are parts of a more general problem from algebraic geometry. 
For simplicity, consider a smooth projective variety $X$ defined over $\C$ and denote by $\Bir(X)$ its  group of birational transformations.
To $f\in \Bir(X)$, one associates its graph in $\Gamma(f)\subset X\times X$. Let $H$ be a polarization of $X$: it can be used to define
a notion of degree for subvarieties of $X\times X$. Now, subvarieties of $X\times X$ of dimension $\dim(X)$ and degree $\leq d$ which are graphs of birational transformations form an algebraic subset of the Hilbert scheme of $X\times X$; let $N^X_{\leq d}$ denote the number of irreducible components of this algebraic variety. By definition, the {\emph{algebraic growth type}} of $\Bir(X)$ is the asymptotic growth type of the sequence $d\mapsto N^X_{\leq d}$. A similar notion can be defined for automorphism groups of affine varieties. Theorems~\ref{ThA} and~\ref{ThB} show that this notion is interesting and non-trivial.  See Section~\ref{par:complements} for further questions

\begin{eg} Let $X$ be a projective surface. Denote by $NS(X;\Z)$ its N\'eron-Severi group. Fix a polarization of $X$ by some ample class $[H]\in NS(X;\R)$ of self intersection $1$, and define the degree of an automorphism $f$ to be the intersection product $(f^*[H]\cdot [H])\in \N$.  Then, $\Aut(X)$ acts on the N\'eron-Severi group $NS(X;\Z)$ and the kernel of this action is an algebraic group (with neutral component denoted by $\Aut(X)^0$). More generally, the index of $\Aut(X)^0$  in the subgroup $\{f\in \Aut(X)\; ; \; f^*[H]=[H]\}$ is finite; we shall denote it by $k_0$. From this, one easily derives that the number of components of $\Aut(X)$ made of automorphisms of degree $\leq d$ is equal to the product $k_0\times M_{\leq d}$ where 
$$
M_{\leq d}=\vert \{  u\in NS(X;\R)\; ; \; u\cdot [H]\leq d {\text{ and }} u\in \Aut(X)^*[H] \} \vert 
$$
Since $\Aut(X)^*$ is a subgroup of $\GL(NS(X;\Z))$ preserving the intersection form (which is of signature $(1,\rho(X)-1)$), one ends up studying a classical problem from hyperbolic geometry that is, the description of the critical exponent of a discrete group of isometries of a hyperbolic space. At the end, one gets examples with bounded, logarithmic, or polynomial growth for $M_{\leq d}$.
\end{eg}

\subsection*{Acknowledgements} Thanks to J\'er\'emy Blanc, Jean-Philippe Furter, and St\'e\-phane Lamy for discussions on irreducible components and homaloidal types. {M. Mella is a member of GNSAGA and has been supported by
  PRIN 2022 project  2022ZRRL4C
  Multilinear Algebraic Geometry.
finanziato dall'Unione Europea - Next Generation EU}. 
S. Cantat and F. Maucourant are supported by the European Research Council (ERC GOAT 101053021). 
   
\section{Preliminaries}

 \vspace{0.2cm}
\begin{center}
\begin{minipage}{12cm}
{\sl{ We introduce some vocabulary that will be used throughout this article and we describe the Hudson's test to provide a new formulation of Theorem~B in terms of homaloidal types. }}
\end{minipage}
\end{center}
\vspace{0.05cm}

\subsection{Homaloidal types} Let $f$ be  a birational transformation of the projective plane $\bbP^2_\bfk$. We can write $f[x:y:z]=[f_0(x,y,z):f_1(x,y,z):f_2(x,y,z)]$ where the $f_i$ are homogeneous polynomials of the same degree $d$ without common factor of positive degree; 
by definition, $d$ is the {\bf{degree}} of $f$. To $f$, one associates its homaloidal net: this is the linear system of curves obtained by pulling back the net of lines by $f$, their equations are 
$af_1+bf_2+cf_3=0$ with $[a:b:c]$ in $\bbP^2_\bfk$. This linear system has degree $d$. We shall denote by  $r$ the number of its base points (including infinitely near base points), by  $p_1$, $\ldots$, $p_r$ the base points, and by  $m_i$ their respective multiplicities. The {\bf{homaloidal type}} of $f$ is the list 
$$ 
\bfx(f)=(d; m_1, \ldots, m_r)
$$ 
where, by convention, the $m_i$ are organized in decreasing order $m_1\geq m_2\geq \cdots \geq m_r$. A second convention is that $r=r(f)$ will denote the number of non-zero multiplicities, but that zeros can always be added at the end of the homaloidal type. For instance, $(2; 1, 1, 1)$ and $(2; 1, 1, 1, 0, 0)$ both represent the homaloidal type of a quadratic birational transformation of the plane, and $r=3$ for such a map. 
In what follows,   we shall also denote by 
$\{ m_1, \ldots, m_r\}$ the unordered list of these numbers {\sl{repeated according to their occurrences}} (for instance $\{ 2, 3, 2\}$ is different from $\{2,3\}$ and is equal to $\{2, 2, 3\}$). 
So, when listing multiplicities $\{ m_i\}$ is not the standard notation usually used for sets.

Grouping multiplicities which are equal, we shall also write a
homaloidal type $\bfx = (d;m_1,\dots,m_r)$ as 
$$
\bfx = (d;\mu_1^{\nu_1},\dots,\mu_s^{\nu_s})
$$
where $\mu_1=m_1$ and $\nu_1$ is the number of occurrences of $m_1$ in $\bfx$, then $\mu_2$ is the largest of the $m_i<m_1$, etc. Thus, the corresponding birational map has $\nu_i$ base points of multiplicity $\mu_i$ for $i = 1,\dots,s$. The number $s = s(\bfx)  \in\mathbb{N}$ will be called the {\bf{seedbed}} of the homaloidal type $\bfx$.
 
A {\bf{block}} of $\bfx = (d;m_1,\dots,m_r)$ is a sequence of equal
multiplicities, and the {\bf{width}} of the block is the number of
elements in this sequence. 

\begin{eg}
The vector
$
\bfx = (d;d-1,1,\dots,1) = (d;d-1,1^{2d-2})
$
is the homaloidal type of de Jonqui\`{e}res maps of degree $d$. It contains two blocks, of  respective widths $1$ and  $2d-2$.
\end{eg}

\subsection{The Noether equalities and inequality} \label{par:Noether}
The homaloidal type $(d;m_1,\dots,m_r)$ of a birational map $f$ satisfies the {\bf{Noether equalities}}
\stepcounter{thm}
\begin{equation}\label{NE}
\sum_{i=1}^r m_i = 3d-3,\quad \sum_{i=1}^r m_i^2 = d^2-1.
\end{equation}
By convention, a homaloidal type will always be obtained from an element of $\Bir(\bbP^2)$(\footnote{As we shall see below, the set of all possible homaloidal types does not depend on the field of definition, provided this field be algebraically closed.}). We will say that $(d;m_1,\dots,m_r)$ is an {\bf{improper homaloidal type}} if it satisfies the Noether equalities but is not the homaloidal type of a birational self-map of~$\mathbb{P}^2$. 

The {\bf{Noether inequality}} tells us that 
\begin{equation}\label{eq:Noether_Inequality}
m_1+m_2+m_3 \geq d+1
\end{equation}
if $(d;m_1, m_2, \ldots, m_r)$ is a (possibly improper) homaloidal type 
with multiplicities $m_1\geq m_2\geq m_3\geq \cdots \geq m_r$ and degree $d\geq 2$ (see~\cite{ac:book}).
In particular, $3m_1>d$ and 
\begin{equation}\label{eq:3m1>d}
m_1 > d/3.
\end{equation}

\subsection{Hudson's test and Hudson's tree}

Let $\bfx = (d;m_1,\dots,m_r)$ be a possibly improper homaloidal type. Hudson's test is an algorithm establishing whether $\bfx = (d;m_1,\dots,m_r)$ is actually a homaloidal type. 
It runs as follows.
\begin{enumerate}

\item  replace $\bfx$ with $\bfx' = (d-\Delta;m_1-\Delta,m_2-\Delta,m_3-\Delta,m_4,\dots,m_r)$, where 
$$\Delta = m_1+m_2+m_3-d.$$ 
(Note that by the Noether inequality, $\Delta$ is positive.)
\item rearrange the multiplicities of $\bfx'$ in non increasing order; 
\item if $\bfx'$ is equal to $(1;0,  \ldots)$ then stop and conclude that $\bfx$ was a homaloidal type; if one of the multiplicities of $\bfx'$ is negative, then stop, and concludes that $\bfx$ was an improper homaloidal type; otherwise, go to step~$1$.  
\end{enumerate}

Geometrically, the first step corresponds to what happens to the homaloidal type when one composes $f$ with a  standard quadratic Cremona involution centered at the three base points of highest multiplicity.  

We now use this algorithm to organize all homaloidal types as the vertices of a tree with root $(1; 0, \ldots)$.

\subsubsection{Parents}

Given a homaloidal type 
$$\bfx=(d;m_1, \ldots,m_r)$$
of degree $d\geq 2$, its {\bf parent} - denoted by $\bfx'$ or by $p(\bfx)$ -   is the homaloidal type
$$\bfx'=(d'; m'_1, \ldots, m'_{r'})$$
obtained as follows. First, one computes 
$$ \Delta(\bfx)= m_1+m_2+m_3-d$$
to get an integer $\Delta(\bfx)\geq 1$. Then, the degree $d'$ of $\bfx'$ is given by 
$$ d'=d-\Delta(\bfx).$$
Then, the multiplicities $m'_\ell$ of $\bfx'$ are obtained in three steps from the multiplicities
of $\bfx$: first, one replaces $m_1$, $m_2$, $m_3$ respectively
$m_1- \Delta(\bfx), m_2-\Delta(\bfx), m_3-\Delta(\bfx)$
and then one reorders the set of multiplicities to list them in decreasing order. Thus, 
there are indices $i$, $j$, $k$ such that 
$$\begin{cases}
m'_i =m_1- \Delta(\bfx) =d-m_2-m_3 \\
m'_j=m_2-\Delta(\bfx)=d-m_3-m_1\\ 
m'_k=m_3-\Delta(\bfx)= d-m_1-m_2
\end{cases}$$
and as a set (with elements repeated according to their multiplicities) we have 
$$ \{m'_1, m'_2 \ldots \}= \{m_1-\Delta(\bfx), m_2-\Delta(\bfx), m_3-\Delta(\bfx), m_4, \ldots, m_r\}$$ 
Our convention is that the indices $(i,j,k)$ will always be chosen as small as possible; in particular, they are uniquely defined. Note that we always have $r'\leq r$ and $d'\leq d-1$. 

We shall also say that $\bfx$ is the {\bf child} of $\bfx'$ obtained
from the {\bf seed} $(m'_i, m'_j, m'_k)$. 

\subsubsection{The tree} Hudson's algorithm tells us that, if we draw
the graph with vertices labeled by homaloidal types and with an edge
between a type and its parent, then we get a tree, with root at $(1;
0)$. 
Going down the tree corresponds to taking sequences of children; when doing so, the degree increases strictly. Going up the tree corresponds to the Hudson's algorithm, or equivalently to computing  the  sequence of  successive parents $\bfx$, $p(\bfx)$, $p(p(\bfx)), $... This tree will be called the {\bf{Hudson tree}}.  A finite sequence $\bfx^0, \ldots,\bfx^N$ of homaloidal types such that $\bfx^n=p(\bfx^{n+1})$ for $n<N$ is called a {\bf lineage} (or an ascending lineage, to say that we go up the tree). 
  

\subsubsection{Children}
     
  Conversely, given a proper homaloidal type 
    $\bfx=(d;m_1, \ldots,m_{r}),$
  and a triple of multiplicities $(m_i,m_j,m_k)$ in non increasing
  order(\footnote{The values may be added zeroes at the end, so that
    some of the multiplicities in the triple can be equal to $0$.}),
  {\sl{with minimal possible indexes}},  one can consider the
  homaloidal type
  $$T_{i,j,k}\bfx:=(d+\nabla(\bfx);\{m_i+\nabla(\bfx),m_j+\nabla(\bfx),m_k+\nabla(\bfx), (m_t)_{t\neq,i,j,k} \} ),$$
  where we ask  the following number to be positive:
\begin{equation} \label{eq:nabla}
  \nabla(\bfx)=d-(m_i+m_j+m_k) \geq 1.
\end{equation}  
We shall also denote this number by $\nabla_{i,j,k}(\bfx)$ to make the triple precise. 

   A triple $(m_i,m_j,m_k)$ is said to be {\bf{admissible}} if it is a seed, i.e.\ if 
   $T_{i,j,k}\bfx$ is a child of $\bfx$. It is equivalent to ask that $(m_i+\nabla(\bfx),m_j+\nabla(\bfx),m_k+\nabla(\bfx))$ are the three largest multiplicities of the set $\{m_i+\nabla(\bfx),m_j+\nabla(\bfx), m_k+\nabla(\bfx), (m_t)_{t\neq i,j,k}\}$. 
   Since $m_i\geq m_j \geq m_k$, this is equivalent to ask that 
  $$\forall t\neq i,j,k, \quad m_k+\nabla(\bfx) \geq m_t,$$
  or equivalently
  \begin{equation}\label{admissible}
   \forall t\neq i,j,k, \quad d-m_i-m_j \geq m_t.
   \end{equation}

 To sum up, 
 \begin{itemize}
 \item any homaloidal type $\bfx$ of degree $d\geq 2$ determines a seed $(m'_i, m'_j, m'_k)$ from the set of multiplicities $m'_1\geq m'_2\geq \ldots \geq m'_{r'}$ of its parent $\bfx'=p(\bfx)$;
 \item conversely, any triple $(m_i,m_j,m_k)$ in the set of multiplicities of $\bfx$ that satisfies~(\ref{admissible}) is a seed of $\bfx$ that determines a child $T_{i,j,k}\bfx$ of~$\bfx$. 
 \end{itemize}
 In what follows,
   we will often consider {\bf{descending lineages}} (also called simply {\bf{descendants}}) $\bfx_0, \ldots,\bfx_N$, which are completely determined by $\bfx_0$ and a sequence of seeds $(m_{i(n)}^{(n)},m_{j(n)}^{(n)},m_{k(n)}^{(n)})$ for $0\leq n\leq N-1$, where 
$$\bfx_n=(d^{(n)};m_{1}^{(n)}, m_{2}^{(n)}, \ldots).$$
    
$$
\begin{tikzpicture}[xscale=1.4,yscale=1.4]
\coordinate (d1) at (4,6);
\coordinate (d2) at (4,5);
\coordinate (d3) at (4,4);
\coordinate (d4a) at (2,3);
\coordinate (d4b) at (6,3);
\coordinate (d5a) at (0,1.5);
\coordinate (d5b) at (4,1.5);
\coordinate (d5c) at (8,1.5);
\coordinate (d6a) at (0,0);
\coordinate (d6b) at (2.667,0);
\coordinate (d6c) at (5.333,0);
\coordinate (d6d) at (8,0);

\draw[red] (d1) -- (d2);
\draw[yellow] (d2) -- (d3);
\draw[red] (d2) -- (d4a);
\draw[black] (d3) -- (d4a);
\draw[yellow] (d3) -- (d4b);
\draw[yellow] (d3) -- (d5b);
\draw[yellow] (d3) -- (d6c);
\draw[red] (d4a) -- (d5a);
\draw[red] (d4a) -- (d6a);
\draw[red] (d4a) -- (d6b);
\draw[black] (d4b) -- (d5b);
\draw[black] (d4b) -- (d5c);
\draw[black] (d4b) -- (d6c);
\draw[black] (d5a) -- (d6a);
\draw[black] (d5b) -- (d6a);
\draw[black] (d5b) -- (d6b);
\draw[black] (d5b) -- (d6c);
\draw[black] (d5c) -- (d6d);

\fill (d1) circle (1pt);
\fill (d2) circle (1pt);
\fill (d3) circle (1pt);
\fill (d4a) circle (1pt);
\fill (d4b) circle (1pt);
\fill (d5a) circle (1pt);
\fill (d5b) circle (1pt);
\fill (d5c) circle (1pt);
\fill (d6a) circle (1pt);
\fill (d6b) circle (1pt);
\fill (d6c) circle (1pt);
\fill (d6d) circle (1pt);

\draw (d1) node [anchor=west,font=\footnotesize]  {$(1;0)$};
\draw (d2) node [anchor=west,font=\footnotesize]  {$(2;1^{3})$};
\draw (d3) node [anchor=west,font=\footnotesize]  {$(3;2,1^{4})$};
\draw (d4a) node [anchor=east,font=\footnotesize]  {$(4;2^{3},1^{3})$};
\draw (d4b) node [anchor=west,font=\footnotesize]  {$(4;3,1^{6})$};
\draw (d5a) node [anchor=east,font=\footnotesize]  {$(5;2^{6})$};
\draw (d5b) node [anchor=east,font=\footnotesize]  {$(5;3,2^{3},1^{3})$};
\draw (d5c) node [anchor=west,font=\footnotesize]  {$(5;4,1^{8})$};
\draw (d6a) node [anchor=north,font=\footnotesize]  {$(6;3^{2},2^{4},1)$};
\draw (d6b) node [anchor=north,font=\footnotesize]  {$(6;3^{3},2,1^{4})$};
\draw (d6c) node [anchor=north,font=\footnotesize]  {$(6;4,2^{4},1^{3})$};
\draw (d6d) node [anchor=north,font=\footnotesize]  {$(6;5,1^{10})$};
\end{tikzpicture}
$$

\begin{eg}
The picture above 
represents all possible ways to reach the homaloidal types in degree $6$, via quadratic Cremona maps, from lower degree homaloidal types. We do not include children in degree $\geq 7$. The red paths represent the complete lineages of $(5;2^6),(6;3^2,2^4,1),(6;4,2^4,1^3)$. They have $(4;2^3,1^2)$ as parent so that $(4;2^3,1^2)$ has $(5;2^6),(6;3^2,2^4,1),(6;4,2^4,1^3)$ has children. The parent of $(4;2^3,1^2)$ is $(2;1^3)$ which has a second child, namely $(3;2,1^4)$. The homaloidal type $(3;2,1^4)$ has in turn three children, namely $(5;3,2^3,1^3)$, $(6,3^3,2,1^4)$, $(4;3,1^6)$. Their complete lineages is represented by the yellow paths together with the red segment joining $(1;0)$ and $(2;1^3)$.
\end{eg}

\subsection{Irreducible components and the main intermediate statement}\label{Bird}  For $d\geq 1$, let $\Bir_d=\Bir(\bbP^2_\bfk)_d$ be the family of all birational transformations $f$ of $\bbP^2_\bfk$ of degree equal to $d$. Using homogeneous formulas $[f_0:f_1:f_2]$ with no common factor, one can endow $\Bir_d$ with the structure of an algebraic variety (embedded as a non-closed variety in $\bbP^{M-1}$ where $M = 3\binom{d+2}{2}$, see~\cite{Blanc-Furter}).

As said in the Introduction, we denote by $N_d$ the number of irreducible components of $\Bir_{d}$. The main result we shall need, beside Hudson's test, is that {\emph{ irreducible components of $\Bir_d$ are in $1$-to-$1$ correspondence with homaloidal types of degree $d$}} (see~\cite{ac:book}). Thus, $N_d$
is just the number of such homaloidal types, and the purpose of this article is to prove the following (as well as a similar result for $\sum_{d'\leq d}N_{d'}$):

\begin{thm} \label{mainthm}
 Let $N_d$ be the number of (proper) homaloidal types of degree $d$, and $N_{\leq d}:=\sum_{d'=1}^D N_{d'}$ the number of homaloidal type of degree $\leq d$. 
Then
 $$\sqrt{\ln(2)}\leq \liminf_{d\to \infty} \frac{\ln \ln N_{\leq d}}{\sqrt{\ln d}}\leq \limsup_{d\to \infty} \frac{\ln \ln N_{\leq d}}{\sqrt{\ln d}}\leq 2\sqrt{\ln 2},$$
 and

$$\sqrt{\ln(2)}\leq \limsup_{d\to \infty} \frac{\ln \ln N_d}{\sqrt{\ln d}} \leq 2\sqrt{\ln 2}.$$
 \end{thm}

The strategy is to estimate precisely the maximal growth of the  seedbed $s(\bfx)$ of a homaloidal type $\bfx$ in terms of its degree $d(\bfx)$. We shall prove the following statement.

\begin{thm} \label{lengthestimate}
For any $\alpha>2\sqrt{\ln2}$, there exists $C_\alpha>0$ such that for all $\bfx$,
$$s(\bfx)\leq C_\alpha e^{\alpha \sqrt{\ln(d(\bfx))}}.$$
For any $\beta<\sqrt{\ln(2)}$, there exists a homaloidal type $\bfx$ of arbitrarily large degree such that
$$s(\bfx)\geq e^{\beta \sqrt{\ln(d(\bfx))}}.$$ 
\end{thm}

As a consequence, 
$$
\sqrt{\ln(2)}\leq \limsup \frac{\ln(s(d))}{\sqrt{\ln(d)}}\leq 2\sqrt{\ln(2)}
$$
where $s(d)$ denotes the {\bf{maximal seedbed}} among homaloidal types of degree~$d$.
The reason why $s(\bfx)$ is called the seedbed of $\bfx$ is because a large seedbed corresponds to a large variety of seeds for $\bfx$, hence a large number of edges emanating from $\bfx$ in the Hudson tree. 

\begin{rem} 
Since any homaloidal type of degree $d$ satisfies the Noether Equalities~\eqref{NE}, $N_d$ is bounded from above by the number of partitions of $3d-3$ and the Hardy-Ramanujam estimate implies that  $N_d\leq  (a/d) \exp(b\sqrt{d}) $, for some $a>0$ and $b=\pi\sqrt{2}$. This upper bound relies only on the first Noether equality; adding the second inequality is reminiscent of the Hilbert-Kamke problem (see~\cite{Karatsuba:Survey}), but with a number $r$ of terms $m_i$ that is not uniformly bounded.
From this, one can show that the number of solutions to~\eqref{NE} grows at least as $\exp(b'\sqrt{d})$  for some $b'>0$. In particular, Theorem~\ref{mainthm} implies that the probability that a solution of~\eqref{NE} be represented by a homaloidal type goes to $0$ as $d$ goes to $+\infty$.The next  table lists the values of $N_d$ and of the number $S_{d}$ of solutions to~\eqref{NE} (with non-increasing $m_i$)
  for $d\in [2,15]$:
$$
\begin{array}{c|c|c|c|c|c|c|c|c|c|c|c|c|c|c}
d & 2 & 3 & 4 & 5 & 6 & 7 & 8 & 9 & 10 & 11 & 12 & 13 & 14 & 15 \\ 
\hline
N_d & 1 & 1 & 2 & 3 & 4 & 5 & 9 & 10 & 17 & 19 & 29 & 34 & 51 & 63\\
\hline
S_{d} & 1 & 1 & 2 & 4 & 5 & 9 & 16 & 25 & 42 & 64 & 107 & 165 & 256 & 402
\end{array} 
$$

\end{rem}

\section{Basic inequalities}

 \vspace{0.2cm}
\begin{center}
\begin{minipage}{12cm}
{\sl{ As recalled in Section~\ref{par:Noether}, the multiplicities of a homaloidal type satisfy certain basic inequalities. The next lemmas describe some of them.  }}
\end{minipage}
\end{center}
\vspace{0.05cm}

 \begin{lem} \label{3first}
Let $(m_i,m_j,m_k)$ be a seed of a homaloidal type of degree $\geq 2$. Then at most one value was chosen from the three largest ones, so that
 $$j > 3.$$ 
 \end{lem}  
 \begin{proof}
  We make a case-by-case analysis to show that any $2$ values taken in the set $\{m_1,m_2,m_3\}$ lead to a contradiction. 
 
 $\bullet$ Assume that the triple is of the form $(m_1,m_2,m_k)$. By Equation \eqref{eq:Noether_Inequality},
   \begin{equation}\label{inequality}
    m_1+m_2+m_3>d,
   \end{equation}
    hence $k\neq 3$; so by Equation (\ref{admissible}), we obtain
   $d-m_1-m_2\geq m_3$,
  in contradiction with Inequality (\ref{inequality}).

$\bullet$Assume that the triple is of the form $(m_1,m_3,m_k)$. By Equation (\ref{admissible}),
   $$d-m_1-m_3\geq m_2$$ 
  and again, this contradicts Inequality (\ref{inequality}).

$\bullet$ Assume that the triple is of the form $(m_2,m_3,m_k)$. By Equation (\ref{admissible}),
   $$d-m_2-m_3\geq m_1$$
   and for the same reason this is a contradiction. 
 \end{proof}

\begin{lem} \label{low_bou_m2m3} 
Let $\bfx=(d;m_1,m_2,m_3, \ldots)$ be a (proper) homaloidal type of degree $d\geq 2$. 
Then
 $$m_1+m_2\leq d, $$ 
 $$ m_1+2m_3> d$$
and 
 $$m_2 \geq \frac{d^2-1-m_1^2}{3d-3-m_1}.$$ 
\end{lem}

The first two inequalities are well known. The second one is in Lemma 8.2.6 of~\cite{ac:book}, it implies Noether's inequality, and it holds for improper homaloidal types such that $m_1+m_2\leq d$. 
\begin{figure}[h]\label{fig:density}
\includegraphics[width=6.5cm]{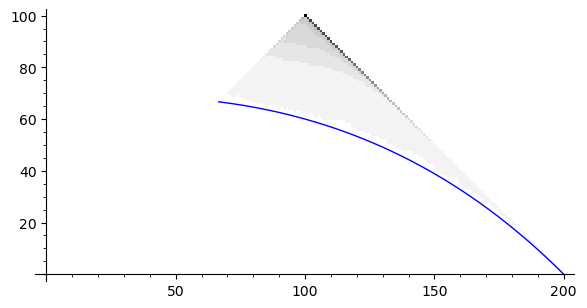}\includegraphics[width=6.5cm]{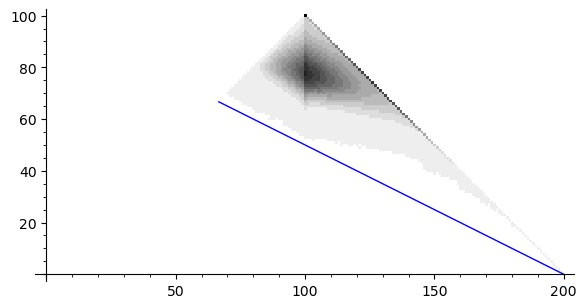}
\caption{On the left, schematic representation of the distribution of ($m_1,m_2$) for the $3585742777$ homaloidal types of degree $200$, where darker indicates more homaloidal types. On the right, the distribution of ($m_1,m_3$) for the same degree. Both are included in triangles determined by the inequalities $m_1+m_2\leq d$, $m_2 \leq m_1$ and similarly for $m_3$. The blue curves indicates the lower bounds on $m_2$ and $m_3$ as functions of $m_1$  given by Lemma~\ref{low_bou_m2m3}.   }
\end{figure}

\begin{proof} Let $\bfx'=(d';m_1', \ldots)$ be the parent of $\bfx$ and let $(m'_i,m'_j,m'_k)$ be the corresponding seed. Then, $m'_k=d-m_1-m_2\geq 0$, hence $m_1+m_2\leq d$.  
For the second inequality, note that $j>3$ by Lemma \ref{3first}. Thus,  $m_4$ is equal to $m'_2$ if $i=1$ or $m'_1$ otherwise. In both cases,
$$m_3\geq m_4 \geq m'_2>m'_j,$$
so 
$$m_3>m_2+(d-m_1-m_2-m_3).$$
Hence the second inequality.

The last inequality is deduced from Noether's equations. By Noether's first equation,
$$d^2-1=m_1^2+m_2^2+\sum_{i\geq 3} m_i^2,$$ 
so since $m_i\leq m_2$ for $i\geq 3$,
$$d^2-1\leq m_1^2+m_2^2+m_2(\sum_{i\geq 3} m_i),$$
using Noether's second equation,
$$d^2-1\leq m_1^2+m_2^2+m_2(3d-3-m_1-m_2)=m_1^2+m_2(3d-3-m_1),$$ 
so
$$\frac{d^2-1-m_1^2}{3d-3-m_1}\leq m_2,$$ 
As required.
\end{proof}

\section{Seeds, $*$-seeds, and tails}

 \vspace{0.2cm}
\begin{center}
\begin{minipage}{12cm}
{\sl{ We collect several important lemmas that  concern the sequences of seeds in a descending lineage. }}
\end{minipage}
\end{center}
\vspace{0.05cm}

\subsection{$*$-triples and $*$-seeds}

 We say that the triple $(m_i,m_j,m_k)$ of $\bfx$ is a {\bf $*$-triple} if $m_i=m_1$ (i.e. $i=1$). 
 Such a triple will often be written $(*,m_j,m_k)$, as the first value is implied. A seed defined by a $*$-triple will be called a {\bf $*$-seed}, and the birth of a child conducted by a $*$-seed will be called a {\bf{$*$-birth}}.
 Of particular interest to us will be lineages that are defined by sequences of successive $*$-seeds, as they admit simple combinatorial descriptions.

 Being a $*$-triple can be characterized by how $d-m_1$ increases:
  \begin{lem} \label{dmm1}
  Let $\bfx=(d;m_1, \ldots,m_r)$ be some homaloidal type, $[m_i,m_j,m_k]$ be a seed and $\bfx'=(d';m_1', \ldots.,m'_r)=T_{i,j,k}\bfx$ be the corresponding child. Then
  $$d'-m'_1\geq d-m_1,$$
  with equality if and only if $(m_i,m_j,m_k)$ is a $*$-seed, that is $m_i=m_1$.
  \end{lem}
   \begin{proof} This follows, with $\nabla=\nabla_{i,j,k}(\bfx)$, from 
   $$ d'-m_1'=(d+\nabla)-(m_i+\nabla)=(d-m_1)+(m_1-m_i)$$
   and the convention that $i=\inf\{\ell\; ; \; m_\ell=m_i\}$.
   \end{proof}

  We will say that the first multiplicity $m_1$ of $\bfx=(d;m_1,m_2, \ldots)$ is {\bf lonely} if $m_1>m_2$. We note the following easy property:
\begin{lem} \label{lonely}
If $\bfx$ is obtained from a $*$-seed of $p(\bfx)$, then its first multiplicity is lonely.
\end{lem}
\begin{proof} Write $p(\bfx)=(d';m'_1,m'_2,  \ldots)$, and let $(m'_1,m'_j,m'_k)$ be the seed giving birth to $\bfx$. Then by Lemma \ref{3first}, $m'_j<m'_2$ so $m'_j<m'_1$. The two largest multiplicities of $\bfx$ being $d'-m'_j-m'_k$ and $d'-m'_1-m'_k$, they are not equal.
\end{proof}

 A $*$-triple $(*,m_j,m_k)$ is a seed if and only if the second value $m_j$ is small enough:
  
 \begin{lem} \label{admissible12}
  A $*$-triple  $(*,m_j,m_k)$ is a $*$-seed if and only if
  $$d-m_1-m_2\geq m_j.$$
 \end{lem} 
 \begin{proof} Assume the triple is a seed, then by Lemma \ref{3first}, $m_2=max_{t\neq 1,j,k} m_t$ so Inequality (\ref{admissible}) proves the claim. Conversely,  if
   $$d-m_1-m_2\geq m_j,$$
   then
     $$d-m_1-m_j\geq m_2 \geq max_{t\neq 1,j,k} m_t$$
     because we always have $m_2\geq max_{t\neq 1,j,k} m_t$.
 So Inequality~\eqref{admissible} is satisfied, i.e.\ the triple $(*,m_j,m_k)$ is a seed.
 \end{proof}
 
\subsection{Tail and successive $*$-seeds} \label{par:definition_of_tail}
  Let $\bfx=(d; m_1, m_2, \cdots, m_r)$ be a homaloidal type. 
Define the {\bf{tail}} of $\bfx$ to be the sequence $(m_\ell, m_{\ell+1}, \ldots, m_r)$ of 
all multiplicities satisfying $m_\ell \leq d-m_1-m_2$. 

The tail of $\bfx$ is related to the parent $\bfx'$ of $\bfx$ in the following way. 
Write $\bfx'=(d'; m'_1, m'_2, \ldots)$ and let $(m'_i, m'_j,m'_k)$ be the corresponding seed. 
Then $m'_k= d-m_1-m_2$ and, in the transition from $\bfx'$ to $\bfx$, the multiplicities $m'_t$ with $t\geq k+1$ 
are kept unchanged in $\bfx$: this sequence of multiplicities $(m'_{k+1}, m'_{k+2}, \ldots)$
coincides with the tail of~$\bfx$.


 



 \begin{lem} \label{decreasing} 
  Let $(m_i,m_j,m_k)$ and $(*,m'_{j'},m'_{k'})$ 
   be two successive triples, for $\bfx_0$ and for its child $\bfx_1$ respectively, the first triple being a seed and the second being a $*$-triple. Then $(*,m'_{j'},m'_{k'})$ is a seed
  \begin{itemize}
  \item if and only if
 $$m_k \geq m'_{j'} \geq m'_{k'},$$
\item if and only if $m'_{j'},m'_{k'}$ are chosen in the tail of $\bfx_1$.  
\end{itemize}
In this case the tail of the child of $\bfx_1$ is obtained from the
tail of $\bfx_1$ by removing the first elements of the sequence up to
$m'_{k'}$.
 \end{lem}
 \begin{proof}
Since $(m_i,m_j,m_k)$ is a seed, the homaloidal type $\bfx_1$ is equal to 
 $$T_{i,j,k}(\bfx_0)=(d+\nabla;m_i+\nabla, m_j+\nabla, m_k+\nabla, m_1, \ldots,\widehat{m_i},  \ldots,\widehat{m_j}, \ldots,\widehat{m_k},  \ldots,m_r)   $$
 where elements with hats are omitted in the list. By Lemma \ref{admissible12}, $(m'_{1},m'_{j'},m'_{k'})$ is a seed if and only if
 $$(d+\nabla)-(m_i+\nabla)-(m_j+\nabla)\geq m'_{j'},$$
 if and only if
 $$d-m_i-m_j - (d -m_i-m_ j-m_k)\geq m'_{j'},$$
 if and only if
\begin{equation} \label{t1adm}
m_k\geq m'_{j'}.
\end{equation}
 Since $\nabla>0$, the value $m_k$ was removed once from the list of $T_{i,j,k}(\bfx_0)$. Thus, the previous inequality means that $m'_{j'}$ and $m'_{k'}$ have to be chosen in the tail
$$m_{k+1}, \ldots, m_r, 0, 0, \ldots$$
of $\bfx_1$, 
because by convention  the index $k$ was chosen as small as possible when the value $m_k$ appears several times in the set of multiplicities.
Conversely, if $m'_{j'},m'_{k'}$ are chosen from the tail of $\bfx_1$, 
Equation (\ref{t1adm}) is satisfied so $(*,m'_{j'},m'_{k'})$ is a seed. In this case,  let $\bfx_2$ be the child of $\bfx_1$ corresponding to this seed; by definition, the tail of $\bfx_2$ is obtained by removing the first elements of the tail of $\bfx_1$ up to the first occurrence of $m'_{k'}$, or the second one if $m'_{j'}=m'_{k'}$.
\end{proof}  
     
 As a consequence, a descending lineage $\bfx_0, \ldots,\bfx_\ell$ obtained by a sequence of $*$-seeds can be encoded by considering the tail $m_{k+1}, \ldots,0,  \ldots$ of $\bfx_0$, by selecting $2\ell$ values in this sequence, say $\mu_1\geq \mu'_1\geq \mu_2\geq \mu'_2 \geq  \ldots\geq \mu'_{\ell} \geq \mu'_{\ell}$, and then by choosing for $\mathbf{x_n}$ ($1\leq n \leq \ell$) the triple $(*,\mu_i,\mu'_i)$. Thus,  there are not so many possibilities for sequences $*$-births, as the tail shorten at each step by at least two elements. 
 
 \begin{eg} If we begin with 
$$\bfx_1=(80;43,31,27,26,26,21,21,18,17,2,2,2,1)$$
and try to list its possible descendants 
by $*$-seeds, we first look at its parent, 
$$\bfx_0=p(\bfx_1)=(59;26,26,22,21,21,18,17,10,6,2,2,2,1),$$ 
and the seed giving birth to $\bfx_1$ is $(22,10,6)$; thus, the remaining tail is $(2,2,2,1,0, \ldots)$.
The choices for $*$-seeds are 
\begin{enumerate} 
\item[(a)] $(*,2,2)$ with remaining tail $(2,1,0, \ldots )$, 
\item[(b)] or $(*,2,1)$,  $(*,2,0)$, $(*,1,0)$, $(*,0,0)$, each with the zero remaining tail $(0, \ldots)$.
\end{enumerate}
Unless we made the first choice $(*,2,2)$, the tail vanishes so the next choices of triple have to be $(*,0,0)$. If we did the first choice $(*,2,2)$, so that
$$\bfx_2=(113;76,35,35,31,27,26,26,21,21,18,17,2,1),$$
then the remaining tail is $(2,1,0 \ldots)$ so the next possible choices are $(*,2,1)$, $(*,2,0)$,
 $(*,1,0)$ and $(*,0,0)$, and each of them reduces the tail to zero, so the next choices of triple have to be $(*,0,0)$. So basically, there is only seven sequences of $*$-seeds originating from $\bfx_1$, which are given by  :
$$(*,2,2),(*,2,1),(*,0,0) \ldots \quad {\mathrm{ and }}  \quad (*,2,2),(*,2,0),(*,0,0) \ldots $$
$$ (*,2,2),(*,0,0)  \ldots    \quad {\mathrm{ and }}  \quad(*,2,1),(*,0,0) \ldots   $$
$$  (*,2,0),(*,0,0) \ldots   \quad {\mathrm{ and }}  \quad  (*,1,0),(*,0,0) \ldots $$
$$(*,0,0) \ldots$$
\end{eg}

\subsection{Successive $*$-triples and degrees}  
 In this section, we consider
 $\bfx_1, \ldots,\bfx_\ell$, a descending lineage such $d(\bfx_1)\geq 2$ and the corresponding (admissible) triples $(*,m_{j(n)}^{(n)},m_{k(n)}^{(n)})$ are all $*$-seeds. First, we can precisely estimate the increase in degree of this sequence, which is roughly linear in $\ell$:

\begin{lem}\label{degree_along_t1}
Let $\bfx_1, \ldots,\bfx_\ell$ be a lineage such that $d(\bfx_1)\geq 2$ and all corresponding seeds are $*$-seeds. Then
$$d(\bfx_\ell)\leq \left( \frac{2\ell+1}{3} \right) d(\bfx_1).$$
If moreover the ratio between the first multiplicity of $\bfx_1$ and its degree is at most $5/6$ then 
$$d(\bfx_\ell)\geq \max\left(\frac{\ell-13}6,1\right) d(\bfx_1).$$
\end{lem}

When $m_1\leq 5d(\bfx_1)/6$, we shall say that  $\bfx_1$ does not have large first multiplicity, see Section~\ref{par:definition_large_small}.
\begin{proof}
 The degree increase between $\mathbf{x_n}$ and $\bfx_{n+1}$ is given by
$$\Delta(\bfx_{n+1})=d(\bfx_{n})-m^{(n)}_1-m^{(n)}_{j(n)}-
m^{(n)}_{k(n)},$$
and by Lemma \ref{dmm1}, the difference between the degree and the first value is constant along a sequence of $*$-births. So
$$\Delta(\bfx_{n+1})=\left( d(\bfx_{1})-m^{(1)}_1 \right) -m^{(n)}_{j(n)}-m^{(n)}_{k^(n)}.$$
Summing over $\ell$ gives  
\begin{align} 
 d(\bfx_\ell)-d(\bfx_1)& =\sum_{n=1}^{\ell-1} \Delta(\bfx_{n+1})\\
\label{eq1} & =(\ell-1)(d(\bfx_{1})-m^{(1)}_1) -\sum_{n=1}^{\ell-1} (m^{(n)}_{j(n)}+m^{(n)}_{k(n)}).
\end{align}
Using that $m^{(1)}_1>d(\bfx_{1})/3$, we get the desired upper bound
$$d(\bfx_\ell)-d(\bfx_1)\leq \frac23(\ell-1)d(\bfx_1).$$
This proves the first inequality; and this 
 is optimal since we can always choose the $m_{k(n)}^{(k)}$ and $m_{k(n)}^{(n)}$ to be $0$.

By the Lemma~\ref{decreasing}, the decreasing sequence
$$m_{j(1)}^{(1)},m_{k(1)}^{(1)},m_{j(2)}^{(2)},m_{k(2)}^{(2)}, \ldots,m_{j(\ell-1)}^{(\ell-1)},m_{k(\ell-1)}^{(\ell-1)}$$
is  extracted from the tail of $\bfx_1$.
Therefore, 
$$\sum_{n=1}^{\ell-1} (m^{(n)}_{j^{(n)}}+m^{(n)}_{k^{(n)}})\leq \sum_{p} m_p^{(1)} < 3d(\bfx_1)
$$
by  Noether's first equation (see Equation~\eqref{NE}). 
By Equation (\ref{eq1}), we obtain
$$d(\bfx_\ell)-d(\bfx_1) > (\ell-1)(d(\bfx_{1})-m^{(1)}_1) - 3d(\bfx_{1}).$$
and since $m_1^{(1)}\leq 5d(\bfx_1)/6$, this gives
$$d(\bfx_\ell)-d(\bfx_1) \geq \left(\frac16(\ell-1)-3 \right)d(\bfx_{1}).$$
Since $d(\bfx_\ell)\geq d(\bfx_1)$, this proves the second inequality.
\end{proof}

\subsection{Large seedbed and large progeny} The following proposition shows that a single homaloidal type with large seedbed $s(\bfx)$ produces many distinct descendants of small degrees. This proposition is at the heart of the proof of the lower bound on $N_d$. We prove it now to illustrate the mechanism of how sequences of $*$-births behave.
 
 \begin{pro} \label{lower_bound_children}
 Let $\bfx$ be a homaloidal type of seedbed $s(\bfx)$ and degree $d$.
Let $s\leq s(\bfx)$ be an integer. Then $\bfx$ has a least $2^s$ distinct descendants of degree $\leq \left(\frac{5s+10}3\right)d$.
 \end{pro}
 \begin{proof}
 Let $m_{i_1}> \ldots >m_{i_{s(\bfx)}}$ be indices corresponding to non-zero distinct multiplicities, chosen minimally.

 \vspace{0.1cm}

\noindent{\bf{Construction of descendants.--}} For each of the $2^s$ subsets $E$ of $\{i_1, \ldots,i_s\}$, we  construct a descendant $\mathbf{y}(E)$, by considering the following sequence of seeds. 

\vspace{0.1cm}
 
{\emph{First 3 (uniform) moves.-- }} The first three choices of triples are made so that further triples will be admissible, while having the tail $m_1, \ldots,m_r$. We begin with
 $$\bfx_1=\bfx=(d;m_1, \ldots ,m_{r}),$$
 and the children we shall construct will be labeled $\bfx_i$, $i=2, 3, \ldots$.
We first use the triple: $(0,0,0)$. It is a seed by Inequality (\ref{admissible}), and  $\Delta(\bfx_2)=d$. We obtain
 $$\bfx_2=(2d;d,d,d,m_1, \ldots,m_{r}),$$
Now, we use the seed $(0,0,0)$, with $\Delta(\bfx_3)=2d$, to get
 $$\bfx_3=(4d;2d,2d,2d,d,d,d,m_1, \ldots,m_{r}).$$
Next, we choose the triple $(d,d,d)$, it is a seed by (\ref{admissible}) because $4d-d-d\geq 2d$, and $\Delta(\bfx_4)=d$; so
 $$\bfx_4=(5d;(2d)^6,m_1, \ldots,m_{r}),$$
where the exponent $6$ indicates repetition. Now the tail of $\bfx_4$ is precisely $(m_1, \ldots,m_{r},0, \ldots)$, and in the next step we start choosing $*$-seeds from this sequence.
 
\vspace{0.1cm}

{\emph{Sequence of moves associated to $E$.-- }}  We choose a subset $E\subset \{i_1, \ldots,i_s\}$, which will act as a parameter, to construct a homaloidal type $\mathbf{y}(E)$.
Consider the values $(m_j)_{j\in E}$; in the case where $E$ has odd cardinality, we add a $0$ at the end of these values. Then we order them
with least possible indices to obtain a subsequence $m_{j_1}>\ldots >m_{j_{2n}}$ of even length, and we consider the sequence of triples
 $$(*,m_{j_1},m_{j_2}),(*,m_{j_3},m_{j_4}), \ldots,(*,m_{j_{2n-1}},m_{j_{2n}}).$$
 By construction,  $m_{j_1},m_{j_2}$ are in the tail of $\bfx_4$, hence  $(*,m_{j_1},m_{j_2})$ is admissible by Lemma \ref{admissible12}.
 Now $(m_{j_3},m_{j_4})$ are also chosen from this tail so $(*,m_{j_3},m_{j_4})$ is, again, an admissible triple for $\bfx_5$. By induction, the whole sequence  of triples
 $$(*,m_{j_1},m_{j_2}), (*,m_{j_3},m_{j_4}), \ldots,(*,m_{j_{2n-1}},m_{j_{2n}})$$
  starting from $\bfx_4$ is a sequence of $*$-seeds and defines a final homaloidal type 
$\mathbf{y}(E):=\bfx_{4+n}$.
 
 \vspace{0.1cm}

\noindent{\bf{Degree estimate, injectivity, and conclusion.--}} First, let us estimate the degree of $\bfy(E)$. By Lemma \ref{degree_along_t1},
$$d(\mathbf{y}(E))\leq \left (\frac{2 n +1}3 \right)d(\mathbf{x_4}).$$
Since $n\leq (s+1)/2$ by construction, we get
$$d(\mathbf{y}(E))\leq \left (\frac{s+2}3\right)5d.$$
Second, let us show that the association $E\mapsto \mathbf{y}(E)$ is injective: given $\mathbf{y}(E)$, we can recover the whole sequence 
  $$(*,m_{j_1},m_{j_2}),(*,m_{j_3},m_{j_4}), \ldots,(*,m_{j_{2n-1}},m_{j_{2n}}),$$
in reverse order by applying Hudson's algorithm to it, until we arrive to $\bfx_4$. Since the numbers $(m_{j_k})_{0\leq k \leq 2n}$ are distinct, the corresponding indices $j_{p},j_{p+1}$ are uniquely determined, and $E$ is the set of such indices.
Since this map is injective, the conclusion follows from the degree estimate. 
\end{proof}

 \section{Constructing homaloidal types of large seedbed}
 
 \vspace{0.2cm}
\begin{center}
\begin{minipage}{12cm}
{\sl{ In this section, we construct homaloidal types of large seedbed, and use them to deduce the lower bound in Theorem~\ref{lengthestimate} and Theorem~B.}}
\end{minipage}
\end{center}
\vspace{0.05cm}

\subsection{Splitting blocks}

Recall that a   block in $(d;m_1, \ldots,m_r)$ is a maximal set of identical nonzero values $m_i=m_{i+1}= \ldots=m_j\neq 0$, and its width is the number of its elements. 

 We will produce homaloidal types of large seedbed by beginning with a de Jonqui\`eres homaloidal type $(2^{N-1}+1;2^{N-1},1^{(2^N)})$, which has one block of  width $2^N$. Such a de Jonqui\'eres homaloidal type is  obtained from the root $(1;0)$ by a series of $2^{N-1}$  $*$-births. 

 The following lemma shows that if we begin with an $\bfx$ that has long blocks, then we can split each block in two while controlling the increase in degree. 

 \begin{lem} \label{splitting_blocks}
Let $k \geq 2$ and $n \geq 1$ be integers. Let $\bfx$ be a proper homaloidal type that contains at least $n$ distinct blocks of width $\geq 2^k$ each. Then there exists a descendant of $\bfx$ of degree $\leq 5  n2^{k}d(\bfx)$ with at least $2n$ distinct blocks of width $\geq 2^{k-1}$.
 \end{lem}
 \begin{proof} Let $d=d(\bfx)$. By assumption, $\bfx$ can be written
 $$\bfx=(d;\square, \ldots,\square,(\mu_1)^{2^k},\square, \ldots,\square,(\mu_2)^{2^k},\square\ldots\square,(\mu_n)^{2^k},\square,\ldots,\square)$$
 where the exponents $2^k$ indicate repetition, with $\mu_1>\mu_2> \ldots>\mu_n$, and $\square$ indicates values that will be ignored. The number of occurrences of one of these $\mu_i$ can, of course, be larger than~$2^k$. As in the proof of Proposition~\ref{lower_bound_children}, we first apply the 3 admissible triples $(0,0,0)$, $(0,0,0)$, $(d,d,d)$ to obtain
$$(5d;(2d)^6,\square, \ldots,\square,(\mu_1)^{2^k},\square, \ldots,\square,(\mu_2)^{2^k},\square\ldots\square,(\mu_n)^{2^k},\square, \ldots,\square).$$
This is done in order to insure that the tail is now the full initial sequence of multiplicities of $\bfx$, at the price of a factor 5 in the degree.
We now apply $(*,\mu_1,\mu_1)$, which is admissible since $(\mu_1,\mu_1)$ is in the tail. We get $\nabla_1=3d-2\mu_1$, and a new homaloidal type
\begin{align*}
(5d+\nabla_1; 2d+\nabla_1, (3d-\mu_1)^{2}, (2d)^5, & \square \ldots\square,(\mu_1)^{2^k-2},\square\ldots\square,(\mu_2)^{2^k}, \\
& \square \ldots\square,(\mu_n)^{2^k},\square \ldots\square)
\end{align*}
and we may continue with $(*,\mu_1,\mu_1)$ until it has been applied $2^{k-2}$ times. Each application has the same $\nabla=\nabla_1$, and we get
\begin{align*}
(5d+2^{k-2}\nabla_1; 2d+2^{(k-2)}\nabla_1, (3d-\mu_1)^{2^{k-1}}, (2d)^5, & \square, \ldots,\square,(\mu_1)^{2^{k-1}}, \\
& \square, \ldots,\square,(\mu_2)^{2^k},\square \ldots)
\end{align*}
Now we do the same with $(*,\mu_2,\mu_2)$, again applied $2^{k-2}$ times.
 This time $\nabla_2=3d-2\mu_2$,  so we get the new homaloidal
 type
\begin{align*}
(  5d+2^{k-2} & \nabla_1+2^{k-2}\nabla_2;  \;  2d+2^{k-2}\nabla_1+ 2^{k-2}\nabla_2,  (3d-\mu_2)^{2^{k-1}},(3d-\mu_1)^{2^{k-1}}, \\
& (2d)^5, 
 \square \ldots\square,(\mu_1)^{2^{k-1}},\square \ldots\square,(\mu_2)^{2^{k-1}},\square \ldots \square,
(\mu_n)^{2^k},\square \ldots \square)
\end{align*}
Continuing this way for the values $\mu_3, \ldots,\mu_n$, we obtain
something of the form
\begin{align*}
(d'; m_1',(3d-\mu_n)^{2^{k-1}}, \ldots,(3d-\mu_1)^{2^{k-1}}, & \square \ldots \square,(\mu_1)^{2^{k-1}},  \\
& \square\ldots\square,(\mu_n)^{2^{k-1}}, \square \ldots \square),
\end{align*}
which has $2n$ distinct blocks of width $\geq 2^{k-1}$. The blocks are distinct since $3d-\mu_n> \ldots>3d-\mu_1>d>\mu_1> \ldots>\mu_n$. 
By Lemma \ref{degree_along_t1}, the degree $d'$ obtained after  this sequence of $\ell-1=n2^{k-2}$ $*$-seeds is at most
$$d'\leq \left(\frac{n2^{k-1}+3}{3}\right)(5d)\leq 5n2^{k}d.$$
  \end{proof}

\subsection{Lower bound on length increase}

  Here we prove the lower bound in Theorem \ref{lengthestimate}, which we recall in the next proposition. 

\begin{pro} \label{lowerboundlength}
 There exists a sequence of proper homaloidal types $(\bfy_N)_{N\geq 1}$ of degrees $d(\bfy_N)\leq 2^{N^2}10^N$, such that $s(\bfy_N)\geq 2^{N-1}$. In particular, for any $c<\sqrt{\ln(2)}$ and $N$ large enough,
$$s(\bfy_N)\geq \exp(c\sqrt{\ln(d(\bfy_N))})$$
\end{pro}
\begin{proof}
 Let $N\geq 1$ be an integer parameter. We begin by constructing recursively an auxiliary finite sequence of homaloidal types $(\bfx_{0,N},\ldots,\bfx_{N-1,N})$ whose last element will be $\bfy_N$.
  If we start with the de Jonqui\`eres homaloidal type 
 $$\bfx_{0,N}:=(2^{N-1}+1; \; 2^{N-1},(1)^{2^N})$$
 which has $1$ block of width $2^{N}$, we can apply Lemma~\ref{splitting_blocks} successively $N-1$ times with $n=2^i$ and $k=N-i$, where $i=0, \ldots,N-2$. Doing so, we  obtain a sequence of homaloidal types 
 $(\bfx_{i,N})=(\bfx_{0,N}, \ldots,\bfx_{N-1,N})$ of degrees $d_{0,N}=2^{N-1}+1, d_{1,N},   \ldots, d_{N-1,N}$ such that each $\bfx_{i,N}$ has $2^{i}$ distinct blocks of width $\geq 2^{N-i}$. The successive degrees satisfy the inequality
 $$d_{i+1,N}\leq 5 \cdot 2^{N-i} \cdot 2^i \cdot d_{i,N}=5 \cdot 2^N\cdot d_{i,N}.$$
 At the end, 
 \begin{align}
 d_{N-1,N} & \leq 5^N 2^{N^2}(2^{N-1}+1)\\
 & \leq 2^{N^2}10^N.
 \end{align}
 We now define the homaloidal type $\bfy_N:=\bfx_{N-1,N}$, so $d(\bfy_N)\leq 2^{N^2}10^N$. Its seedbed $s(\bfy_{N})$ is larger than the number $2^{N-1}$ of blocks of width $2$ that have been created.
The last statement follows from the fact that when $N$ is large, we have
 \begin{equation}
 N\geq \sqrt{\frac{\ln(d(\bfy_N))}{\ln(2)}} \left(1+o(1)\right),
 \end{equation} 
 and the seedbed satisfies
 \begin{equation}
  s(\bfy_{N})\geq 2^{N-1}=\exp(N\ln(2)+O(1)).
 \end{equation}
 \end{proof}
 
\subsection{Lower bound on the number of proper homaloidal types}\label{par:proof_of_lower_bound_Nd}
 
 Combining Proposition~\ref{lower_bound_children} and~\ref{lowerboundlength}, we now prove the lower bound of Theorem \ref{mainthm}.
\begin{thm} 
Let $N_d$ be the number of proper homaloidal types of degree $d$, and $N_{\leq d}$ the number of proper homaloidal types of degree $\leq d$.
Then
 $$\sqrt{\ln(2)}\leq \liminf_{d\to \infty} \frac{\ln \ln N_{\leq d}}{\sqrt{\ln d}},$$
and for every $\beta<\sqrt{\ln(2)}$, 
there exists arbitrary large degrees $d$ for which
 $$\ln \ln N_d\geq \beta \sqrt{\ln d}.$$
\end{thm}
 \begin{proof}
  By Proposition~\ref{lowerboundlength}, there is a sequence of homaloidal types $\bfy_k$ of degree $d_k \leq 2^{k^2}10^k$, with seedbed $s_k\geq 2^{k-1}$. We choose the auxiliary parameter $s=2^{k-1}$, and apply  Proposition~\ref{lower_bound_children}. So $\bfy_k$ has at least $2^s$ children of degree $\leq D_k:=\left(\frac{5s+10}{3}\right)d_k$, so as $k$ tends to infinity,
  $$\ln(D_k)\leq \ln(s)+O(1)+\ln(d_k)\leq k^2 \ln 2 + k\ln(20)+O(1).$$
 so
 $$\sqrt{\ln(D_k)}\leq k \sqrt{\ln 2 +o(1)},$$

Let $\beta<\sqrt{\ln(2)}$, then for $k$ large enough, say $k\geq k_\beta$ for some $k_\beta>0$, we have
\begin{equation}\label{eq:ineqD_k_k}
\sqrt{\ln(D_k)} \leq k \frac{\ln(2)}\beta.
\end{equation}
 Let $d$ be a large integer. and choose the unique integer $k$ (that depends now on $d$) such that
 \begin{equation}\label{eq:def_d_k}
 \frac{\beta}{\ln(2)}\sqrt{\ln(d)}-1 < k \leq  \frac{\beta}{\ln(2)}\sqrt{\ln(d)}.
 \end{equation}
 Then if $d$ is large enough, $k$ is larger than $k_\beta$, so by \eqref{eq:ineqD_k_k} and \eqref{eq:def_d_k}, we have $D_k\leq d$.
 In particular, remembering that $\bfy_k$ has at least $2^{2^{k-1}}$ children of degree $\leq D_k$, we get
 $$2^{2^{k-1}}\leq N_{\leq D_k}\leq N_{\leq d}.$$
 Using \eqref{eq:def_d_k}, we obtain
 \begin{align}
 \ln \ln N_{\leq d} & \geq (k-1)\ln(2)+\ln \ln(2)\\
 & \geq \beta \sqrt{\ln(d)}+O(1).
 \end{align}
 Therefore, as $d$ tends to $+\infty$, we have
 $$\liminf_{d\to +\infty} \frac{\ln \ln N_{\leq d}}{\sqrt{\ln(d)}}\geq \beta,$$
 which is valid for all $\beta<\sqrt{\ln(2)}$, thus implying the lower bound:
 \begin{equation}\label{eq2}
 \liminf_{d\to +\infty} \frac{\ln \ln N_{\leq d}}{\sqrt{\ln(d)}}\geq \sqrt{\ln(2)}.
 \end{equation}

Now assume, by contradiction, that there exists $\beta<\sqrt{\ln(2)}$ such we have $\ln\ln N_d< \beta \sqrt{\ln d}$ for all $d$ sufficiently large, say $d\geq d_1$. Pick $c$ such that $\beta<c<\sqrt{\ln(2)}$.
 Then for $D$ sufficiently large so that $\max_{d<d_1} N_d<\exp(\exp(\beta \sqrt{\ln D}))$, we would have
\begin{align*}
 N_{\leq D}   & \leq D \exp\left( \exp(\beta \sqrt{\ln D}) \right),\\
 N_{\leq D} & =o\left(\exp\left(\exp(c\sqrt{\ln D}\right)\right),
\end{align*}
because $\beta < c$, and this would contradict the lower bound \eqref{eq2}. 
 \end{proof}

\subsection{Upper bounds on the seedbed increase} 
 In the previous paragraphs, we managed to produce homaloidal types of large seedbed; to do this, we could double the seedbed, at each step, along a sequence of $*$-births. The main point of this section is to prove that through such a sequence of $*$-births, the seedbed can at most be multiplied by $4$.  In the proof of Lemma \ref{splitting_blocks}, we could notice that with  a $*$-seed $(*,\mu,\mu)$ we might have created a new multiplicity value; however, applying the same type of seed $(*,\mu,\mu)$ a second time did not: it enlarged the width of an existing block, but did not create a new multiplicity. More formally, we have the following lemma. 

\begin{lem} \label{successive}
Let $\bfx_1,\bfx_2,\bfx_3$ be a lineage obtained from $\bfx_1$ by application of two  successive $*$-seeds  $(*,\mu,\mu)$, for some multiplicity $\mu$. Then
$$s(\bfx_3)\leq s(\bfx_2).$$
\end{lem}
\begin{proof} 
We denote
$$\bfx_1=(d;m_1,m_2,m_3, \ldots, (\mu)^{4}, \ldots)$$
where the block of value $\mu$ is of width at least $4$. Since the assumption is that the triples $(*,\mu,\mu),(*,\mu,\mu)$ are seeds, we know that $(\mu)^{4}$ must be part of the tail of $\bfx_1$ (see Lemma~\ref{decreasing}). 
So, applying $(*,\mu,\mu)$, we get the first child
$$\bfx_2=(2d-m_1-2\mu; \; d-2\mu,(d-m_1-\mu)^{2},m_2,m_3, \ldots,(\mu)^{2}, \ldots)$$
Recall that by Lemma \ref{lonely}, the first multiplicity is then lonely.
We apply $(*,\mu,\mu)$ again, assuming this is an admissible triple. The increase in degree is the same as in the previous step, so we get
$$\bfx_3=(3d-2m_1-2\mu; \; d-2\mu,(d-m_1-\mu)^{4},m_2,m_3, \ldots),$$
where the first multiplicity is still lonely, and no new multiplicity value is created (but, maybe, the multiplicity $\mu$ disappeared). So in any case, $s(\bfx_3)\leq s(\bfx_2)$.
\end{proof}


\begin{pro}\label{s_bound_t1}
Let $\bfx_1, \ldots,\bfx_\ell$ be a descending lineage with $d(\bfx_1)\geq 2$, such that all corresponding seeds are $*$-seeds.
Then 
$$s(\bfx_\ell)\leq \min\left(4s(\bfx_1)-1, \, s(\bfx_1)+3(\ell-1) \right).$$
\end{pro}

{\noindent}{\bf{Heuristical Remark.}} The factor $4$ in  $s(\bfx_\ell)\leq 4s(\bfx_1)-1$ is directly related to the constant $2\sqrt{\ln(2)}=\sqrt{2 \ln(4)}$ in the upper bound given in Theorem~B. 
Heuristically, imagine we could produce homaloidal types in the spirit of Proposition \ref{lowerboundlength} by constructing a sequence $\bfy_0,\ldots, \bfy_n,\ldots $ of respective degrees $d_0, \ldots$,  $d_n, \ldots$ and seedbeds $s_0, \ldots, s_n, \ldots$, where 
$\bfy_{n+1}$ is obtained from $\bfy_n$ by a small, bounded number of birth followed by a long sequence of $*$-birth, of length $\ell_n$. Let's assume that the worst case scenario of seedbed increase is when the two upper bounds given by Proposition \ref{s_bound_t1} are equal, that is when $\ell_n\simeq s_n$, and that it is realized by our sequence $(\bfy_n)_n$. In this case, we would have $s_{n+1}\simeq 4s_n$, so $s_n \simeq 4^n$, and by Lemma \ref{degree_along_t1}, $d_{n+1}\simeq \ell_n d_{n}$ so $d_n \simeq 4^{\frac{n(n+1)}2}$. In this heuristic computation, we would get
$$\lim_{n\to +\infty} \frac{\ln(s_n)}{\sqrt{\ln(d_n)}}=\sqrt{2\ln(4)}.$$
Our task later will be to explain that this heuristic is, somehow, the worst case scenario authorized by Proposition \ref{s_bound_t1}, namely to show that for homaloidal types $\bfx$ of large degrees,
$$\limsup_{d(\bfx)\to +\infty} \frac{\ln(s(\bfx))}{\sqrt{\ln(d(\bfx))}}\leq\sqrt{2\ln(4)}.$$
Improving the factor $4$ would improve this upper bound.
More precisely, what may happen is that when one concatenates $m$ descending lineages that start and end in the region of average first multiplicity but otherwise are always given by $*$-births  in the region of large first multiplicity, then the obvious factor $4^m$ could possibly be replaced by a smaller quantity.

\begin{proof}[Proof of Proposition~\ref{s_bound_t1}]
 Denote by $\left(*,m_{j(n)}^{(n)},m_{k(n)}^{(n)}\right)$, $1\leq n\leq \ell-1$, the corresponding $*$-seeds. 
By Lemma \ref{decreasing}, the non-increasing sequence
\begin{equation}\label{eq:first_equation_of_proof}
m_{j(1)}^{(1)},m_{k(1)}^{(1)},m_{j(2)}^{(2)},m_{k(2)}^{(2)}, \ldots,m_{j(\ell-1)}^{(\ell-1)},m_{k(\ell-1)}^{(\ell-1)}
\end{equation}
is an extracted subsequence of the tail of $\bfx_1$. 

Let $\alpha_1> \ldots>\alpha_t$ be the values that appear in the sequence~\eqref{eq:first_equation_of_proof}. This means that the successive $*$-seeds must be of the form
$(*,\alpha_i,\alpha_i)$ or $(*,\alpha_i,\alpha_{i+1})$, and the indices must be nondecreasing. Note that each triple $(*,\alpha_i,\alpha_{i+1})$ can happen only once; and it may happen that $(*,\alpha_i,\alpha_i)$ appears several times, but always consecutively.
We have the obvious bound $t\leq s(\bfx_1)+1$ (taking care of the possibility that $\alpha_t=0$), but in this can be improved to $t\leq s(\bfx_1)$ because the third multiplicity $m_3$ of $\bfx_1$ is never in the tail of $\bfx_1$.

 At each of the $\ell-1$ steps, we change the values of three multiplicities, so clearly
$$s(\bfx_\ell)\leq s(\bfx_1)+3(\ell-1).$$
 
 When applying a triple of the form $(*,\alpha_i,\alpha_{i+1})$, one adds at most three new multiplicities. However, by Lemma \ref{lonely}, the first multiplicity is lonely whenever $n\geq 2$, so apart possibly from the first one, seeds of the form $(*,\alpha_i,\alpha_{i+1})$ add at most two new multiplicities each. When we apply successively $(*,\alpha_i,\alpha_{i})$ any number of times, by Lemma \ref{successive}, we may have added only one multiplicity the first time and none the others, unless it was $(*,\alpha_1,\alpha_1)$ and the first multiplicity wasn't lonely in $\bfx_1$. In the end,
$$s(\bfx_\ell)-s(\bfx_1)\leq 2(t-1)+t+1,$$
the final $+1$ counting the  possibility that $\bfx_1$ did not have a lonely first multiplicity. 
Since $t\leq s(\bfx_1)$, 
we obtain $s(\bfx_\ell)-s(\bfx_1)\leq 3s(\bfx_1) -1$, as desired.
\end{proof}
 
 From the above upper bound, we now derive an inequality on the seedbed increase that will be more suitable for iterative applications.

\begin{cor} \label{ineq_t1}
For any $\alpha>\sqrt{2\ln(4)}$, there exist $S_0(\alpha)>0$ and $L(\alpha)\geq 5$,  such that
\begin{equation}\label{fundamental_inequality}
(\ln s(\mathbf{x_{\ell+1}}) )^2-(\ln s(\mathbf{x_{1}}) )^2\leq \alpha^2 \left( \ln d(\bfx_{\ell+1})-\ln d(\bfx_{1}) \right)
\end{equation}
 for any descending lineage $\bfx_1, \ldots,\bfx_\ell,$ $\bfx_{\ell+1}$ such that
 \begin{enumerate} 
\item[(i)]  $s(\bfx_1)\geq S_0(\alpha)$, $\ell\geq L(\alpha)$, and $d(\bfx_1)\geq 2$; 
\item[(ii)]   $\bfx_1$ does not have large first multiplicity, i.e. $m_1(\bfx_1)\leq 5d(\bfx_1)/6$;
\item[(iii)]  all the corresponding seeds are $*$-seeds, except possibly the last one. 
 \end{enumerate}
 \end{cor}
 \begin{rem} The inequality $L(\alpha)\geq 5$ is arbitrary, we choose it because it will be used in the proof of Proposition~\ref{upp_b_s}. 
 \end{rem}
\begin{proof} Fix $\alpha>2\sqrt{\ln2}$ and define $\epsilon>0$   by 
$$2\ln(4) +2\epsilon=\alpha^2.$$
Set
$\delta'= \ln d(\bfx_{\ell+1})-\ln d(\bfx_{1}) .$
By Lemma \ref{degree_along_t1},
$$\ln d(\bfx_{\ell+1})\geq \ln d(\bfx_{\ell}) \geq \ln \left( \frac{\ell-13}{6} \right) + \ln d(\bfx_1),$$
so
\begin{equation}\label{eq:estimate_ell_delta'}
\ln \left(\frac{\ell-13}{6} \right) \leq \delta'.
\end{equation}

The increase of the seedbed is captured by the quantity
\begin{align*}
\delta& = (\ln s(\mathbf{x_{\ell+1}}) )^2-(\ln s(\mathbf{x_{1}}) )^2  \\
&=  \ln \left(\frac{s(\mathbf{x_{\ell+1}})}{s(\mathbf{x_{1}})}\right) \times \ln \left(s(\mathbf{x_{\ell+1}}) s(\mathbf{x_{1}}) \right)
\end{align*}
and the goal is to prove that $\delta\leq \alpha^2 \delta'$. We write 
$$
 \ln \left(s(\mathbf{x_{\ell+1}}) s(\mathbf{x_{1}}) \right)=2\ln\left( \frac{\ell-13}6\right)
+ 2\ln  \left(\frac{6\ell}{\ell-13}   \right)
  + 2 \ln \left(\frac{s(\mathbf{x_{1}})}{\ell} \right)
+ \ln\left(  \frac{s(\mathbf{x_{\ell+1}})}{s(\mathbf{x_{1}})} \right).
$$
and obtain
\begin{align*}
\delta & =  2 \ln \left(\frac{s(\mathbf{x_{\ell+1}})}{s(\mathbf{x_{1}})}  \right) \cdot \ln \left(\frac{\ell-13}6 \right)
+ 2 \ln \left(\frac{s(\mathbf{x_{\ell+1}})}{s(\mathbf{x_{1}})}  \right) \cdot \ln \left(\frac{6\ell}{\ell-13} \right)  \\
& \;  + 2 \ln\left( \frac{s(\mathbf{x_{\ell+1}})}{s(\mathbf{x_{1}})}  \right) \cdot \ln \left(\frac{s(\mathbf{x_{1}})}{\ell} \right)
+ \left( \ln \frac{s(\mathbf{x_{\ell+1}})}{s(\mathbf{x_{1}})} \right)^2.
\end{align*}
By Proposition~\ref{s_bound_t1},  
$$s(\bfx_\ell)\leq \min\left(4s(\bfx_1)-1), \, s(\bfx_1)+3(\ell-1) \right), $$
so 
$$\begin{cases}
s(\bfx_{\ell+1})\leq 4s(\bfx_1)+2 \\
s(\bfx_{\ell+1})\leq  s(\bfx_1)+3\ell. 
\end{cases}$$
Since $s(\mathbf{x_{1}})\geq 1$, we have the upper bound $s(\bfx_{\ell+1})/s(\bfx_1)\leq 6$. If $\ell\geq 26$, then $\frac{6\ell}{\ell-13}\leq 12$.
So, with Inequality~\ref{eq:estimate_ell_delta'}, we obtain
$$\delta  \leq 
2 \ln \left(4+\frac{2}{s(\mathbf{x_{1}})}\right) \cdot \delta'
+ 2 \ln 6  \cdot  \ln 12  \\
 \;  + 2 \ln \left(1+\frac{3\ell}{s(\mathbf{x_{1}})}\right)  \cdot  \ln \frac{s(\mathbf{x_{1}})}{\ell}
+ \left( \ln 6 \right)^2.$$
The function $f(x)=2\ln(1+3x)  \cdot \ln(1/x)$ is bounded from above on $\mathbb{R}^+$ by a constant $M<+\infty$ (numerically, $M<1.6$). Hence, if $M'=M+2\ln 6 \cdot \ln 12+(\ln 6)^2$, we obtain
$$\delta  \leq 2 \ln \left(4+\frac{2}{s(\mathbf{x_{1}})}\right)  \cdot \delta' + M'.$$

  Let $L(\alpha)>26$ be such that for all $\ell\geq L(\alpha)$,
$$M'<\epsilon \ln \left(\frac{\ell-13}{6}\right) $$
Then, by the Estimate~\eqref{eq:estimate_ell_delta'}
\begin{equation} \label{ineq1}
\delta \leq 2 \ln \left(4+\frac{2}{s(\mathbf{x_{1}})}\right) \cdot \delta'+\epsilon\times  \delta'.
\end{equation}
We choose now $S_0(\alpha)\geq 1$ such that 
$$2\ln \left(4+\frac{2}{S_0(\alpha)} \right)\leq \epsilon +2\ln 4.$$
Then
$$\delta \leq (2\ln 4 +2\epsilon) \delta',$$
as required. 
\end{proof}

\section{Small, average, and large first multiplicity}

 \vspace{0.2cm}
\begin{center}
\begin{minipage}{12cm}
{\sl{ In this section, we introduce three regimes: small, average, and large first multiplicities. We study how the degree and seedbed vary when a descending lineage stays in one of these regimes. }}
\end{minipage}
\end{center}
\vspace{0.05cm}

\subsection{Definition}\label{par:definition_large_small}

 Consider a homaloidal type $\bfx=(d;m_1, \ldots)$ of degree $d\geq 2$.
Note that the ratio $m_1/d$ is always between $1/3$ and $1$ (see Equation~\eqref{eq:3m1>d}). We say that $ \bfx$
has {\bf large first multiplicity} if
$$m_1/d> 5/6.$$ We say that $\bfx$ has {\bf small first multiplicity} if
$$m_1/d<7/20.$$

Otherwise, the ratio is in $[7/20,5/6]$, 
and we say that $\bfx$ has {\bf average first multiplicity}. Observe that $(2;1,1,1)$ has average first multiplicity and is a common ancestor for all nontrivial homaloidal types.
The goal of this section is to describe different behaviors related to the parent of $\bfx$ according to the size of the first multiplicity.

\begin{rem} The cutoff values $5/6$ and $7/20$ are rather arbitrary choices. See for instance the discussion around the proof of Lemma~\ref{smallmult}. \end{rem}

\subsection{Large first multiplicity}
\begin{lem} \label{largefirst} 
Let $\bfx=(d; \; m_1, \ldots)$ be a homaloidal type of large first multiplicity. Denote by $\bfx'=(d';\; m_1', \ldots)$ its parent and by $(m'_i,m'_j,m'_k)$  the corresponding seed. Then
\begin{itemize}
 \item the seed $(m'_i,m'_j,m'_k)$ is a $*$-seed, and
 \item the first multiplicity of $\bfx'$ is not small (except if $d=2$ and $\bfx'=(1;0)$).
\end{itemize}
\end{lem}

\begin{proof} We may assume $d\geq 3$, so that $d'\geq 2$ too. 
Then, by Lemma~\ref{low_bou_m2m3}, $m_1+m_2\leq d$ and $m_1'+m_2'\leq d'$. Since $m_1>5d/6$, we get $m_2<d/6$ and since $m_3 \leq m_2$, we get $m_3<d/6$. So the degree increase between $\bfx'$ and $\bfx$ is at most 
$$\Delta(\bfx)=m_1+m_2+m_3-d\leq d + d/6 + d/6-d =d/3.$$
This implies that
$$m'_i=m_1-\Delta(\bfx)\geq 5d/6-d/3=d/2.$$
Assume $i\neq 1$, i.e.\  $(m'_i,m'_j,m'_k)$ is not a $*$-seed. In this case $m'_1\geq m'_2\geq m'_i$, so $m'_1+m'_2\geq 2m'_i\geq d > d' $. This contradicts the inequality $m'_1+m'_2\leq d'$, and we conclude that $(m'_i,m'_j,m'_k)$ is indeed a $*$-seed.

Now, the second assertion follows from  
$$\frac{m'_1}{d'}=\frac{m_1-\Delta(\bfx)}{d-\Delta(\bfx)}\geq
\frac{d/2}{d}=\frac{1}{2}.$$

\end{proof}

\subsection{Small first multiplicity}


\begin{lem}  \label{smallmult} 
Let $\bfx=(d;m_1, \ldots)$ be a homaloidal type of small
first multiplicity. Denote by $\bfx'=(d';m_1', \ldots)$ its parent and by $(m'_i,m'_j,m'_k)$ the corresponding seed. Then 
\begin{itemize}
 \item the seed was chosen among the first $11$ multiplicities of $\bfx'$, i.e.\ $k\leq 11$.
\item the multiplicity of $\bfx'$ is not large.
\end{itemize}
\end{lem}
We shall prove a slightly stronger statement. Indeed, assume we replace $7/20$ by $17/50$ in the definition of small first multiplicity; then, $k\leq 11$ could be replaced by $k\leq 9$ in the lemma, and this is optimal. 

\begin{proof} Let $\epsilon$ be a positive real number $\leq 1/20$, and assume that $m_1< (1+\epsilon)d/3$. With $\epsilon=1/20$ the assumption becomes $m_1<7d/20$.

Then, $m_3\leq m_2\leq m_1<(1+\epsilon)d/3$ and 
$$\Delta(\bfx)=m_1+m_2+m_3-d< \epsilon d.$$
However, $m_1+m_2+m_3>d$, so $2m_1+m_3>d$ and then $m_3 > (1-2\epsilon)d/3$ (so, in particular, the 3 first multiplicities must be close to $d/3$). Now 
$$m'_k=m_3-\Delta(\bfx)> (1-2\epsilon)d/3 - \epsilon d= (1-5\epsilon)d/3.$$
By Noether's first equation $\sum_\ell m'_\ell =3d'-3\leq 3d,$ hence the sum over all multiplicities 
$m'_\ell$ from $m'_1$ to $m'_k$ is strictly less that $3d$ but larger than $k\times (1-5\epsilon)d/3$. This gives 
$$ 
k(1-5\epsilon) < 9.
$$ 
With $\epsilon = 1/20$, we obtain $k\leq 11$, and for $\epsilon \leq 1/50 $ we obtain $k\leq 9$.
This proves the first assertion of the lemma. 

For the second assertion, we choose $\epsilon = 1/20$ and estimate $m'_1/d'$: 
$$\frac{m'_1}{d'}=\frac{m_1-\Delta(\bfx)}{d-\Delta(\bfx)}\leq
\frac{21d/60}{d-d/20}=\frac{7}{19}<5/6,$$
so $\bfx'$ does not have large first multiplicity.
\end{proof}

\begin{rem}
 By Lemma~\ref{largefirst} and Lemma~\ref{smallmult}, when we look at a lineage, then the first multiplicity cannot go from large to small or small to large directly without passing through  an average first multiplicity.
 \end{rem}

\subsection{Average first multiplicity}

 This can be considered somewhat the "good" case, where an exponential growth in degree is observed; the precise  result we obtained is summarized by the following proposition.
\begin{pro} \label{average}
There exists $\eta_0>0$ such that for any  homaloidal type
$\bfx$  of average first multiplicity and degree $\geq 2$, we have
$$d(\bfx)\geq (1+\eta_0) d(\bfx')$$
where $\bfx'$ is its parent.
\end{pro}

\begin{proof}[Proof of Proposition \ref{average}] If we denote by $x$ the ratio $m_1/d$ of the homaloidal type $\bfx=(d;m_1,m_2,m_3, \ldots)$ then by Lemma \ref{low_bou_m2m3}
$$ \frac{\Delta(\bfx)}{d}=\frac{m_1+m_2+m_3-d}{d}\geq x+\frac{1-x^2-1/d^2}{3-x-3/d}+\frac{1-x}2-1,$$
$$ \frac{\Delta(\bfx)}{d}\geq \frac{x-1}2+\frac{1-x^2-1/d^2}{3-x},$$
because $1/(3-x)\leq 1/2$, so
$$ \frac{\Delta(\bfx)}{d}\geq \frac{-3x^2+4x-1}{2(3-x)} -\frac{1}{2d^2},$$
$$ \frac{\Delta(\bfx)}{d}\geq \frac{(1-x)(3x-1)}{2(3-x)} -\frac{1}{2d^2}.$$
Clearly, the map $x\mapsto \frac{(1-x)(3x-1)}{2(3-x)}$  has a strictly positive lower bound on $(7/20,5/6)$, say $\epsilon>0$. For $d$ sufficiently large, $\frac{1}{2d^2}<\epsilon/2$, so 
$$\frac{\Delta(\bfx)}{d}\geq \epsilon/2.$$
Since
$$d(\bfx')=d(\bfx)-\Delta(\bfx)\leq (1-\epsilon/2) d(\bfx).$$
This implies the result provided $d(\bfx)$ is large enough. Now, we can adjust the constant $\eta_0$ to fit the inequality for the  finitely many remaining cases of low degree.
\end{proof}

\subsection{Sequences of small highest multiplicity}

\begin{lem} \label{seqsmall}
Consider a descending lineage $\bfx_1, \ldots,\bfx_N$, all of which have small first multiplicity, and let  $\bfx_0$ be the parent of $\bfx_1$. Then
$$s(\bfx_N)\leq s(\bfx_0)+11.$$
\end{lem}
\begin{proof} By Lemma \ref{smallmult}, for $n=0, \ldots, N-1$ the seed of $\bfx_n$ used to construct $\bfx_{n+1}$ is always made of multiplicities taken among the eleven first. Thus the multiplicities $m_k$ with $k\geq 12$  are the same all along the lineage and the seedbed may only vary by separating values which were equal between the first $12$ ones. Hence the bound.
\end{proof}

\subsection{Suitable bound on seedbed increase}

 The  following proposition will be useful later to deduce from an additive inequality  $s(\bfx_\ell)\leq s(\bfx_1)+C$, like in  Lemma \ref{seqsmall}, an inequality of the form (\ref{fundamental_inequality}).

\begin{pro}  \label{fundamental_inequality2}
Given $\alpha>\sqrt{2\ln 4}$, there exists $S_1(\alpha)>0$ with the following property.
Let $\bfx_1, \ldots,\bfx_\ell$ ($\ell\geq2$) be a descending lineage such that
\begin{itemize}
\item $s(\bfx_1)\geq S_1(\alpha)$,
\item $\bfx_\ell$ has average  first multiplicity,
\item $s(\bfx_\ell)\leq s(\bfx_1)+3L(\alpha)$, where $L(\alpha)$ was given in Corollary~\ref{ineq_t1}.
\end{itemize}
Then
$$(\ln s(\mathbf{x_{\ell}}) )^2-(\ln s(\mathbf{x_{1}}) )^2\leq \alpha^2 \left( \ln d(\bfx_{\ell})-\ln d(\bfx_{1}) \right).$$
\end{pro}
\begin{proof} This is similar to the proof of Corollary~\ref{ineq_t1}, but much cruder. Consider
\begin{align}
\delta' & =\ln d(\bfx_\ell)-\ln d(\bfx_1) \\
& \geq \ln d(\bfx_{\ell})-\ln d(\bfx_{\ell-1}).
\end{align}
Since $\bfx_{\ell}$ has average first multiplicity, Proposition \ref{average} tells us that this difference is bounded from below $\ln(1+\eta_0)$, hence
$$\delta' \geq \ln(1+\eta_0).$$
Now we set $\delta=(\ln s(\mathbf{x_{\ell}}) )^2-(\ln s(\mathbf{x_{1}}) )^2$ and we wish to prove that $\delta\leq \alpha^2 \ln(1+\eta_0)$, since then we can conclude that $\delta\leq \alpha^2 \delta'$. For this, write
\begin{align*}
\delta& = (\ln s(\mathbf{x_{\ell}}) )^2-(\ln s(\mathbf{x_{1}}) )^2  \\
&=  \ln \left(\frac{s(\mathbf{x_{\ell}})}{s(\mathbf{x_{1}})}\right) \cdot \ln \left(s(\mathbf{x_{\ell}}) s(\mathbf{x_{1}})\right)  \\
& =  2 \ln \left( \frac{s(\mathbf{x_{\ell}})}{s(\mathbf{x_{1}})} \right) \cdot \ln s(\mathbf{x_{1}})
+\left( \ln \left(\frac{s(\mathbf{x_{\ell}})}{s(\mathbf{x_{1}})} \right)\right)^2 \\
& \leq 2 \ln \left(1+\frac{3L(\alpha)}{s(\mathbf{x_{1}})} \right) \cdot  \ln s(\mathbf{x_{1}}) + \left(\ln \left(1+\frac{3L(\alpha)}{s(\mathbf{x_{1}})}\right)\right)^2 \\
& \leq 6L(\alpha)\frac{\ln s(\mathbf{x_{1}})}{s(\mathbf{x_{1}})} + \left(\frac{3L(\alpha)}{s(\mathbf{x_{1}})}\right)^2,
\end{align*}
where we used $\ln(1+x)\leq x$. Now, the function 
$$f(x)=6L(\alpha)\frac{ \ln x}{x}+ \left(\frac{3L(\alpha)}{x}\right)^2,$$ 
goes to $0$ when $x$ goes to $+\infty$, so there exists an integer $S_1(\alpha)\geq 1$ such that $f(x)< \alpha^2 \ln(1+\eta_0)$ when $x\geq S_1(\alpha)$. With such a choice, 
$$\delta < \alpha^2 \ln(1+\eta_0),$$
provided $s(\bfx_1)\geq S_1(\alpha)$, as required.
\end{proof}

\section{Bound on the length}

 \vspace{0.2cm}
\begin{center}
\begin{minipage}{12cm}
{\sl{ In this section, we  prove the upper bounds in Theorems \ref{lengthestimate} and~\ref{mainthm}. Altogether, this will conclude the proof of Theorem~B.}}
\end{minipage}
\end{center}
\vspace{0.2cm}

\subsection{Upper estimate on the seedbed growth} The next proposition recalls the first assertion of Theorem \ref{lengthestimate}.

\begin{pro} \label{upp_b_s}
For all $\alpha>\sqrt{2\ln (4)}$, there exists $C_\alpha>0$ such that for every homaloidal type $\bfx$, we have
$s(\bfx)\leq C_\alpha \exp(\alpha \sqrt{\ln d(\bfx)}).$
\end{pro}
\begin{proof} 

Fix $\alpha>2\sqrt{\ln 2}$, and set 
$$
S_2(\alpha)=4\max(S_0(\alpha), S_1(\alpha))+14,
$$
where $S_0(\alpha)$ and $S_1(\alpha)$ are given by Corollary~\ref{ineq_t1} and Proposition \ref{fundamental_inequality2} respectively.  We will always require $C_\alpha\geq S_2(\alpha)$.

\vspace{0.1cm}

\noindent{\bf{Step 1. Preliminary reductions--}} Let $\bfx$ be a homaloidal type of degree $\geq 2$. 
If $s(\bfx)<S_2(\alpha)$, we are done, so {\sl{we can assume that $s(\bfx)\geq S_2(\alpha)$}}.
We can also {\sl{assume that $\bfx$ has average first multiplicity and that the seed  corresponding to $p(\bfx)$ is not a $*$-seed}}. Indeed, if it is not the case, we choose the seed $(0,0,0)$ of $\bfx=(d;m_1,m_2, \ldots)$ and consider the corresponding child $\bfx'=(2d;d,d,d,m_1,m_2, \ldots).$
If we manage to prove that 
$$s(\bfx')\leq C'_\alpha \exp \left(\alpha \sqrt{\ln d(\bfx')} \right)$$
for some $C'_\alpha>0$ that depends only on $\alpha$, then
$$s(\bfx)=s(\bfx')-1\leq C'_\alpha \exp \left( \alpha \sqrt{\ln (2d(\bfx))} \right).$$
Then, $d\geq 2$ implies successively that $\ln(d)>0.6$ and 
$$\sqrt{\ln(d)+\ln (2)}\leq \sqrt{\ln(d)}+\ln(2),$$ 
thus we obtain
$$s(\bfx)\leq (C'_\alpha e^{\alpha \ln2}) \exp(\alpha \sqrt{\ln d(\bfx)}),$$
which proves the proposition with $C_\alpha=C'_\alpha e^{\alpha \ln2}$.

\vspace{0.1cm}

\noindent{\bf{Step 2.--}} According to the first step, we can now assume that $\bfx$ has average first multiplicity, the seed of $p(\bfx)$ is not a $*$-seed, and $s(\bfx)\geq S_2(\alpha)$. Let's consider the whole descending  lineage from $\bfx_1=(2;1,1,1)$ up to $\bfx_N=\bfx$, and  let's cut this sequence as follows.

We extract from $\bfx_{1}, \ldots,\bfx_N$ the subsequence $(\mathbf{y}_n)_{1\leq n\leq p}$ made  of  the ancestors of $\bfx$ with  average first multiplicities. Hence $p\geq 2$ since $\bfy_1=(2;1,1,1)$ and $\bfy_p=\bfx$.  We define $k\leq p$ to be the smallest index of this sequence such that, for all $n$ between $k+1$ and $q$, we have
$$s(\bfy_n)\geq S_2(\alpha).$$ 
By definition,
$$s(\bfy_k)< S_2(\alpha).$$ 

 Our goal in this second step is to show that there is no big jump in the sequence $s(\bfy_n)$, more precisely, for $n=1, \ldots,p-1$,
\begin{equation} \label{no_big_jumps}
s(\bfy_{n+1})\leq 4s(\bfy_{n})+14.
\end{equation}

 Let's look at the lineage between $\bfy_n$ and $\bfy_{n+1}$, say $\bfx_q=\bfy_n, \ldots,\bfx_{q+\ell}=\bfy_{n+1}$.
Since the sequence $\bfx_n$ cannot jump from large to small first multiplicity or vice-versa in one birth, by definition of the sequence $\bfy_n$ as the subsequence of average first multiplicities, there are only three cases to consider:
\begin{itemize}
\item $\ell=1$, that is $\bfy_{n+1}$ is the children of $\bfy_{n}$. In this case,
$$s(\bfy_{n+1})\leq s(\bfy_{n})+3.$$
\item  $\bfx_{q+1}, \ldots,\bfx_{q+\ell-1}$ all have small first multiplicity. By Lemma \ref{seqsmall}, we have
$$s(\bfx_{q+\ell-1})\leq s(\bfy_{n})+11,$$
so 
$$s(\bfy_{n+1})\leq s(\bfy_{n})+14.$$
\item  $\bfx_{q+1}, \ldots,\bfx_{q+\ell-1}$ all have large first multiplicity. Then by Lemma \ref{largefirst}, $\bfy_n=\bfx_{q}$, $\ldots$, $\bfx_{q+\ell-1}$ must be a lineage entirely made of $*$-births, so by Proposition \ref{s_bound_t1},
$$s(\bfx_{q+\ell-1})\leq 4s(\bfy_{k})-1;$$
thus,
$$s(\bfy_{n+1})\leq 4s(\bfy_{n})+2,$$
which again implies Inequality (\ref{no_big_jumps}).
\end{itemize}

\vspace{0.1cm}

\noindent{\bf{Step 3.--}} It follows from Inequality (\ref{no_big_jumps}) and the definition of the index $k$ that 
$$4s(\bfy_k)+14\geq s(\bfy_{k+1})\geq S_2(\alpha)=4\max(S_0(\alpha),S_1(\alpha))+14,$$
so $s(\bfy_k)\geq \max(S_0(\alpha),S_1(\alpha))$. Then, this inequality remains valid for all indices $n>k$:
$$\forall n\geq k, \, s(\bfy_n)\geq \max(S_0(\alpha),S_1(\alpha)).$$
This lower bound will be required to apply Corollary~\ref{ineq_t1} or Proposition \ref{fundamental_inequality2}.

\vspace{0.1cm}

 We now prove that for $n=k, k+1, \ldots,p-1$, we have the inequality
\begin{equation} \label{desired_ineq}
(\ln s(\bfy_{n+1}))^2-(\ln s(\bfy_{n}))^2\leq \alpha^2 \left(\ln d(\bfy_{n+1})-\ln d(\bfy_{n})\right).
\end{equation}
For this, we look at the lineage between $\bfy_n$ and $\bfy_{n+1}$, namely 
 $\bfx_q=\bfy_n$, $\ldots$, $\bfx_{q+\ell}=\bfy_{n+1}.$ So, $\ell$ is the number of births from $\bfy_n$ to $\bfy_{n+1}$ in Hudson's tree.
  
 If  $\bfy_n$ and $\bfy_{n+1}$ are parent and child, or if the intermediate lineage consists of homaloidal types all of small multiplicity, as we observed before,
$$s(\bfy_{n+1})\leq s(\bfy_{n})+14\leq s(\bfy_{n})+3L(\alpha),$$
because $L(\alpha)\geq 5$ by definition. Therefore, Proposition \ref{fundamental_inequality2} applies and inequality (\ref{desired_ineq}) follows.

 If the intermediate lineage consists of homaloidal types all of large first multiplicity, we distinguish between the two cases $\ell \leq L(\alpha)$ and $\ell> L(\alpha)$. In the first case, since $s(\bfy_{n+1})\leq s(\bfy_{n})+3\ell\leq s(\bfy_{n})+3L(\alpha)$, we may apply Proposition \ref{fundamental_inequality2} to get Inequality \eqref{desired_ineq}. In the second case, Corollary~\ref{ineq_t1} applies, and again \eqref{desired_ineq} is satisfied.
 
\vspace{0.1cm}

\noindent{\bf{Step 4. Conclusion--}} Thus, summing the Inequalities \eqref{desired_ineq} from $n=k$ to $p-1$, we obtain
$$\left( \ln s(\bfx))^2-(\ln s(\bfy_{k}) \right)^2\leq \alpha^2 \left(\ln d(\bfx)-\ln d(\bfy_{k})\right).$$
So, since by definition of $k$, $s(\bfy_k)< S_2(\alpha),$ 
$$(\ln s(\bfx))^2 \leq \alpha^2 \ln d(\bfx) +(\ln s(\bfy_{k}))^2.$$
This gives
\begin{align*}
\ln s(\bfx) & \leq \sqrt{\alpha^2 \ln d(\bfx) + (\ln S_2(\alpha))^2} \\
& \leq \alpha \sqrt{\ln d(\bfx)} + \ln S_2(\alpha),
\end{align*}
because $\sqrt{a+b}\leq \sqrt{a}+\sqrt{b}$ for all $a,b\geq 0$.
Thus $ s(\bfx) \leq S_2(\alpha)\exp(\alpha \sqrt{\ln d(\bfx)})$, as desired.
\end{proof}

\subsection{Upper estimate on $N_d$} We can now prove the upper bound on $N_d$ from Theorem~\ref{mainthm}.

\vspace{0.1cm}

 Let $d\geq 1$ be a  large integer, and let $X_d$ be the set of (proper) homaloidal types of degree $d$. 
 Let $\alpha>2\sqrt{\ln 2}$,  and $\sigma$ the smallest integer $>C_\alpha e^{\alpha\sqrt{\ln d}}$, where $C_\alpha$ is the constant in Proposition~\ref{upp_b_s}.
We assume that $d$ is large enough to satisfy $\sigma<d$.

 By Proposition \ref{upp_b_s}, if  $\bfx \in X_d$ then $s(\bfx)\leq \sigma$, thus $\bfx$ can be written as a series of blocks
$$\bfx=(d;(\mu_1)^{\nu_1}, \ldots,(\mu_\sigma)^{\nu_\sigma}),$$
where  blocks of zeros may be added at the end.
We consider the map $\pi:X_d \to \N^\sigma$ defined by 
$$\pi\colon (d;(\mu_1)^{\nu_1}, \ldots,(\mu_\sigma)^{\nu_\sigma})\mapsto (p_1, \ldots,p_\sigma)=(\mu_1\nu_1, \ldots,\mu_\sigma \nu_\sigma).$$

 By the first Noether equation, the image of $\pi$ is a subset of the set of non-negative integers $(p_1, \ldots,p_\sigma)$ such that $\sum_{i=1}^\sigma p_i=3d-3$. Therefore, the image of $\pi$ is in a set of cardinality
$$\binom{\sigma-1+3d-3}{\sigma-1}\leq (4d)^\sigma.$$
We can bound the number of elements of each fiber of $\pi$: given $p_1, \ldots,p_\sigma$ in $\N$, then if $p_i=0$ we can take $\mu_i=0$ too, and if $p_i\neq 0$, $p_i=\mu_i\nu_i$ is a decomposition of the integer $p_i\leq 3d$ into the product of two integers $\geq 1$; there are at most $3d$ such decompositions for each $i$, thus,
$$|\pi^{-1}(p_1, \ldots,p_\sigma)|\leq (3d)^{\sigma}.$$
Thus $N_d=|X_d|  \leq  (3d)^{\sigma} (4d)^\sigma$, hence $\ln N_d  \leq  \sigma \ln(12d^2)$, and then 
\begin{align*}
\ln \ln N_d & \leq \ln \sigma + \ln \ln (12d^2),\\
& \leq \ln (C_\alpha+1) +\alpha \sqrt{\ln d} + \ln \ln (12d^2)
\end{align*}
Dividing by $\sqrt{\ln(d)}$, this gives 
$$\frac{\ln \ln N_d}{\sqrt{\ln d}}  \leq \alpha + o(1)$$
when $d\to+\infty$, as was to be proved.

\subsection{Upper estimate on $N_{\leq d}$} 
The upper bound on $N_{\leq d}$ from Theorem~\ref{mainthm} is deduced as follows from the upper bound on $N_d$. Let $\alpha>2\sqrt{\ln(2)}$, choose a smaller exponent $\alpha'$ such that $\alpha>\alpha'>2\sqrt{\ln(2)}$. Then for $d$ large enough, say $d\geq d_0$, we have
$$N_d\leq \exp(\exp(\alpha'\sqrt{\ln d})).$$
Assume $d$ is large enough so that for the finitely many $d'\leq d_0$, we have
$$N_{d'} \leq \exp(\exp(\alpha'\sqrt{\ln d})).$$
Thus 
$$N_{\leq d}=\sum_{d'=1}^d N_d\leq d\exp(\exp(\alpha'\sqrt{\ln d})).$$
so, since $\alpha'<\alpha$,
$$N_{\leq d}=o(\exp(\exp(\alpha\sqrt{\ln d})).$$

\section{Questions}\label{par:complements}


\subsection{Monotonicity} We don't know whether $(N_d)$ is increasing (in the range $1\leq d\leq 249$, it is). 
Theorem~\ref{ThA} shows that closely related examples provide oscillating sequences. Similar oscillations occur for finitely generated groups  (see~\cite{Kassabov-Pak}). 

\subsection{Finitely generated subgroups of $\Bir(\bbP^2_\bfk)$} Let $G$ be a subgroup of $\Bir(\bbP^2_\bfk)$. For $d\geq 1$, denote by $c_G(d)$ the number of components of $\Bir(\bbP^2_\bfk)_d$ containing at least one element of $G$; thus, $c_G(d)\leq N_d$ with equality when $G$ is equal to $\Bir(\bbP^2_\bfk)$, or when ${\mathrm{char}}(\bfk)=0$ and $G$ contains $\Bir(\bbP^2_\Q)$ (this follows from Hudson's algorithm). On the other hand, we do not know whether a finitely generated subgroup of $\Bir(\bbP^2_\bfk)$ can visit all components of $\Bir(\bbP^2_\bfk)$. 
As a more general problem, one can ask for a description of the possible growth rates of $(c_G(d))_{d\geq 1}$ as $G$ varies among all (resp.\ all finitely generated) subgroups of $\Bir(\bbP^2_\bfk)$ or when $G$ is some explicit subgroups, for instance the group of monomial transformations or the subgroup of $\Bir(\bbP^2_{ {\mathbf{F}}_p})$ generated by the standard quadratic involution and $\PGL_3({\mathbf{F}}_p)$.

\subsection{The minimal and maximal dimensions} Let  $\bfx=(d; m_1, \ldots, m_r)$ be a homaloidal type and denote by $r(\bfx)=r$ the number of non-zero multiplicities of $\bfx$ (with repetition). The dimension of the irreducible component  $I_x$ of $\Bir(\bbP^2_\bfk)$ corresponding to $\bfx$ is equal to 
\begin{equation}
\dim(\bfx)=8+2r(x)
\end{equation}
because the general point of $I_x$ is determined by its $r(\bfx)$ base points up to post composition by an element of $\PGL_3(\bfk)$ and $\dim(\PGL_3)=8$. The component of maximal dimension is unique, it corresponds to the Jonqui\'eres type $(d; d-1, 1^{2d-2})$, and its dimension is $6+4d$.
The minimal possible dimension is obtained with $r(\bfx)=9$ and is equal to $26$. Using Halphen surfaces and their automorphisms, it is easy to show that this minimal dimension occurs for a subset of $\N$ of positive density; in fact, numerical simulations suggest that it occurs for all large enough $d$. 
It would be great to have a description of the asymptotic shape of the discrete curve 
\begin{equation}
D(m)={\mathrm{card}}\{ \bfx\; ; \; \bfx \; {\text{is a homaloidal type of degree}} \; d \; {\text{with }} \dim(\bfx)=m\}.
\end{equation}

\subsection{Densities} Figure~\ref{fig:density} suggests that, for a homaloidal type of degree $d$ taken with probability $1/N_d$, the vector $(x,y,z)=\frac{1}{d}(m_1,m_2,m_3)$ determines a random variable in $\R^3$ which, for $d$ large, equidistributes towards a probability measure~$\mu$. It would be great to find the exact asymptotic of $N_d$ and then show that such a measure exists.

\subsection{Other varieties} We don't know what is the growth rate for the number of irreducible components $N_d^{\bbP^m}$ of $\Bir(\bbP^m_\bfk)_d$. We don't know if there is a variety $X$ for which $N_{\leq d}^X\simeq \ln(\ln(d))$ or another one for which $N_{\leq d}^X\simeq \exp(d)$.

\bibliographystyle{plain}
\bibliography{references}

\begin{thebibliography}{1}

\bibitem{ac:book}
Maria Alberich-Carrami\~nana.
\newblock {\em Geometry of the plane {C}remona maps}, volume 1769 of {\em
  Lecture Notes in Mathematics}.
\newblock Springer-Verlag, Berlin, 2002.

\bibitem{Blanc-Furter}
J\'er\'emy Blanc and Jean-Philippe Furter.
\newblock Length in the {C}remona group.
\newblock {\em Ann. H. Lebesgue}, 2:187--257, 2019.

\bibitem{deleglise-hernane-nicolas}
M.~Del\'eglise, M.~O. Hernane, and J.-L. Nicolas.
\newblock Grandes valeurs et nombres champions de la fonction arithm\'etique de
  {K}alm\'ar.
\newblock {\em J. Number Theory}, 128(6):1676--1716, 2008.

\bibitem{friedland-milnor}
Shmuel Friedland and John Milnor.
\newblock Dynamical properties of plane polynomial automorphisms.
\newblock {\em Ergodic Theory Dynam. Systems}, 9(1):67--99, 1989.

\bibitem{Karatsuba:Survey}
A.~A. Karatsuba.
\newblock A {H}ilbert-{K}amke problem in analytic number theory.
\newblock {\em Mat. Zametki}, 41(2):272--284, 288, 1987.

\bibitem{Kassabov-Pak}
Martin Kassabov and Igor Pak.
\newblock Groups of oscillating intermediate growth.
\newblock {\em Ann. of Math. (2)}, 177(3):1113--1145, 2013.

\end{thebibliography}
\vfill
\pagebreak
\end{document}